\documentclass[letter,12pt,reqno,twoside]{amsart}

\usepackage[left=1in,right=1.5in,top=1in,bottom=1in]{geometry}
\usepackage{amsmath,amsfonts,amsthm,mathrsfs,amssymb,bm,cite}
\usepackage{graphicx,epstopdf,booktabs}
\usepackage{array,paralist,verbatim}
\usepackage{caption,subcaption,cases}
\usepackage[usenames]{color}
\usepackage[colorlinks=true,urlcolor=blue,linkcolor=blue,citecolor=red]{hyperref}
\usepackage[textsize=tiny,textwidth=1in,shadow,backgroundcolor=yellow!40]{todonotes}

\newtheorem{thm}{Theorem}[section]

\newtheorem{lem}{Lemma}[section]
\newtheorem{prop}{Proposition}[section]
\theoremstyle{definition}
\newtheorem{defn}{Definition}[section]
\theoremstyle{remark}
\newtheorem{rem}{Remark}[section]
\numberwithin{equation}{section}

\newcommand{\RR}{\mathbb{R}}
\newcommand{\Sphere}{\mathbb{S}}

\newcommand{\curl}{{\rm curl} \,}

\title[Recovering polyhedral PEC obstacles]{\bf Recovering an electromagnetic obstacle by a few phaseless backscattering measurements}

\author{Jingzhi Li}
\address{Faculty of Science, South University of Science and
Technology of China, 518055 Shenzhen, P.~R.~China.}
\email{li.jz@sustc.edu.cn}

\author{Hongyu Liu}
\address{Department of Mathematics, Hong Kong Baptist University,Kowloon Tong, Hong Kong SAR.\vspace*{-4mm}}
\address{\vspace*{-4mm}and}
\address{HKBU Institute of Research and Continuing Education, Virtual University Park, Shenzhen, P. R. China.}
\email{hongyu.liuip@gmail.com}

\author{Yuliang Wang}
\address{Department of Mathematics, Hong Kong Baptist University,Kowloon Tong, Hong Kong SAR.}
\email{yuliang@hkbu.edu.hk}

\date{}

\begin{document}
\maketitle

\begin{abstract}

We consider the electromagnetic scattering from a convex polyhedral PEC or PMC obstacle due to a time-harmonic incident plane wave. It is shown that the modulus of the far-field pattern in the backscattering aperture possesses a certain local maximum behavior. Using the local maximum indicating phenomena, one can determine the exterior unit normal directions, as well as the face areas, of the front faces of the obstacle. Then we propose a recovery scheme of reconstructing the obstacle by phaseless backscattering measurements. This work significantly extends our recent study in \cite{LL} from two dimensions and acoustic scattering to the much more challenging three dimensions and electromagnetic scattering.

\medskip

 \smallskip


\end{abstract}

\section{Introduction}

In this paper, we shall be concerned with the inverse scattering problem of recovering an anomalous obstacle located in a homogeneous space by the corresponding electromagnetic (EM) wave probing. In doing so, one sends a certain electromagnetic wave field, and the wave propagation will be interrupted/perturbed when meeting with the obstacle. The perturbation is the so-called scattering, and one intends to recover the obstacle by measuring the scattered wave away from the obstacle. The inverse scattering problem is of fundamental importance to many areas of science and technology, including radar/sonar, geophysical exploration, medical imaging, as well as remote sensing; see \cite{AK1,AK2,CK,Isa,Uhl} and the references therein.

Let $\epsilon$ and $\mu$ denote, respectively, the electric permittivity and magnetic permeability of the isotropic homogeneous medium in $\mathbb{R}^3$. Let $D\subset\mathbb{R}^3$ be a bounded Lipschitz domain such that $\mathbb{R}^3 \setminus \overline{D}$ is connected. Here $D$ represents the target obstacle located in the homogeneous space. The electromagnetic wave is descried by the electric field $\mathscr{E}(\bm{x}, t)$ and the magnetic field $\mathscr{H}(\bm{x},t)$ for $(\bm{x},t)\in\mathbb{R}^3\times\mathbb{R}_+$. The electromagnetic wave propagation is governed by the Maxwell equations
\begin{equation}\label{eq:Maxwell1}
\curl \mathscr{E}(\bm{x},t)+\mu\frac{\partial\mathscr{H}}{\partial t}(\bm{x},t)=0,\quad
\curl \mathscr{H}(\bm{x},t)-\epsilon\frac{\partial \mathscr{E}}{\partial t}(\bm{x},t)=0.
\end{equation}
For time-harmonic electromagnetic waves of the form
\[
\mathscr{E}(\bm{x},t)=\Re(\epsilon^{-1/2}\bm{E}(\bm{x})e^{-i\omega t}),\quad \mathscr{H}(\bm{x},t)=\Re(\mu^{-1/2}\bm{H}(\bm{x})e^{-i\omega t})
\]
with frequency $\omega\in\mathbb{R}_+$, it is directly verified that one has the reduced Maxwell equations
\begin{equation}\label{eq:Maxwell2}
  \curl \bm{E}(\bm{x})-i k \bm{H}(\bm{x})=0,\quad 
  \curl \bm{H}(\bm{x})+ik \bm{E}(\bm{x})=0,
\end{equation}
where $k=\omega \sqrt{\epsilon\mu}\in\mathbb{R}_+$ denotes the wavenumber. The EM wave field cannot penetrate inside the obstacle $D$, and hence the Maxwell system \eqref{eq:Maxwell2} is defined only in $\mathbb{R}^3\setminus \overline{D}$, and on the boundary $\partial D$, one has
\begin{equation}\label{eq:PEC}
\bm{\nu}(\bm{x})\times \bm{E}(\bm{x})=0\quad\mbox{or}\quad \bm{\nu}(\bm{x})\times \bm{H}(\bm{x})=0,\quad \bm{x}\in\partial D,
\end{equation}
respectively, corresponding to a perfectly electric conducting (PEC) obstacle or a perfectly magnetic conducting (PMC) obstacle. Here, $\bm{\nu}\in\Sphere^2:=\{\bm{x}\in\mathbb{R}^3; |\bm{x}|=1\}$
denote the exterior unit normal vector to $\partial D$. The total wave fields $(\bm{E}(\bm{x}), \bm{H}(\bm{x}))$ are given as
\begin{equation}\label{eq:total}
\bm{E}(\bm{x})=\bm{E}^i(\bm{x})+\bm{E}^s(\bm{x}),\quad \bm{H}(\bm{x})=\bm{H}^i(\bm{x})+\bm{H}^s(\bm{x}),\quad \bm{x}\in\mathbb{R}^3\setminus\overline{D},
\end{equation}
where for the present study, we take
\begin{equation}\label{eq:planewave}
  \bm{E}^i = e^{ik\bm{x}\cdot \bm{d}} \bm{p},\quad 
  \bm{H}^i(\bm{x}) = e^{ik\bm{x}\cdot \bm{d}} (\bm{d} \times \bm{p}).
\end{equation}
Here $(\bm{E}^i, \bm{H}^i)$ in \eqref{eq:planewave} are known as the normalized electromagnetic plane wave with the polarization vector $\bm{p}\in\Sphere^2$ and incident direction $\bm{d}\in\Sphere^2$ satisfying $ \bm{d} \cdot \bm{p} = 0$, while $\bm{E}^s$, $\bm{H}^s$ in \eqref{eq:total} are known as the scattered electric and magnetic fields, respectively, and they are required to satisfy the Silver-M\"uller radiation condition as follows,
\begin{equation}\label{eq:radiation}
\lim_{|\bm{x}|\rightarrow+\infty}(\bm{H}^s\times \bm{x}-|\bm{x}| \bm{E}^s)=0,
\end{equation}
which holds uniformly for all directions $\hat{\bm{x}}:=\bm{x}/|\bm{x}|$, $\bm{x}\in\mathbb{R}^3$ and $\bm{x}\neq 0$. The Maxwell system \eqref{eq:Maxwell2}--\eqref{eq:radiation} is well understood and there exists a unique pair of solutions $(\bm{E}, \bm{H})\in H_{\rm loc}(\mbox{curl}; \mathbb{R}^3\setminus\overline{D})\times H_{\rm loc}(\mbox{curl}; \mathbb{R}^3\setminus\overline{D})$ (cf. \cite{CK,Ned}) such that as $|\bm{x}|\rightarrow +\infty$,
\begin{equation}\label{eq:farfield}
\bm{E}(\bm{x})=\frac{e^{ik|\bm{x}|}}{|\bm{x}|} \bm{E}^\infty(\hat{\bm{x}})+\mathcal{O}\left(\frac{1}{|\bm{x}|^2}\right),\quad \bm{H}(\bm{x})=\frac{e^{ik|\bm{x}|}}{|\bm{x}|} \bm{H}^\infty(\hat{\bm{x}})+\mathcal{O}\left(\frac{1}{|\bm{x}|^2}\right),
\end{equation}
which hold uniformly for all directions $\hat{\bm{x}}$. Here $\bm{E}^\infty$ and $\bm{H}^\infty$ are known as the electric and magnetic far-field patterns, respectively, and they satisfy
\begin{equation}\label{eq:relation}
\bm{H}^\infty=\hat{\bm{x}}\times \bm{E}^\infty\quad \mbox{and}\quad \hat{\bm{x}}\cdot \bm{E}^\infty=\hat{\bm{x}}\cdot \bm{H}^\infty=0.
\end{equation}
In what follows, we shall write $\bm{E}^\infty(\hat{\bm{x}}; \bm{p}, k, \bm{d}, D)$ to specify its dependence on the observation direction $\hat{\bm{x}}$, polarization $\bm{p}$, wavenumber $k$, incident direction $\bm{d}$ and the obstacle $D$.

The inverse scattering problem that we are concerned with is to recover $D$ by the knowledge of $\bm{E}^\infty(\hat{\bm{x}}; \bm{p}, k, \bm{d}, D)$. The inverse problem is widely known to be nonlinear and ill-posed (cf. \cite{CK}). There is a longstanding problem in the literature on whether and how one can recover the obstacle $D$ by using a single far-field measurement; that is, $\bm{E}^\infty(\hat{\bm{x}}; \bm{p}, k, \bm{d})$ given for all $\hat{\bm{x}}\in\Sphere^2$ but fixed $\bm{p}, k$ and $\bm{d}$ (see \cite{CK,Isa,Uhl}). Physically speaking, a single far-field measurement is obtained by sending a single incident plane wave and then collecting the electric far-field data in every observation direction. It is remarked that $\bm{E}^\infty$ (respectively, $\bm{H}^\infty$) is a real-analytic function on $\Sphere^2$, and hence if it is known on any open patch of the unit sphere, then it is known on the whole sphere by the analytic continuation (cf. \cite{CK}). It is easily seen that the inverse problem is formally posed with a single far-field measurement. Hence, there is a widespread belief that one can establish the recovery by a single far-field measurement, though it still remains to be a very challenging issue. We refer to \cite{KR,LLSS,LLW,Liu,LYZ} for some theoretical and computational progress on the investigation of the recovery for the inverse electromagnetic scattering problem by making use of as few measurement data as possible. Another extremely challenging issue for the inverse scattering problem is about the recovery by the phaseless data, say the modulus of the electric far-field pattern, $|\bm{E}^\infty(\hat{\bm{x}})|$. To our best knowledge, there is very little progress in the literature on the phaseless recovery for the inverse electromagnetic scattering problem described above.

In our recent work \cite{LL}, a novel scheme was developed for the reconstruction of a polyhedral obstacle by a few acoustic backscattering measurements. The scheme is based on the high-frequency asymptotics of the acoustic wave scattering, namely the Kirchhoff or the physical optics approximation. Using the high-frequency asymptotic approximation, it is shown in \cite{LL} that the modulus of the acoustic far-field pattern in the backscattering aperture possesses a certain local maximum behavior, from which one can determine the exterior normal directions of the front faces of the obstacle. Then by a few backscattering measurements corresponding to several properly chosen incident plane waves, one can determine the exterior normal directions of the faces of the obstacle. After the determination of the exterior face normals, the recovery of the whole obstacle is reduced into a finite dimensional algebraic problem, which can be easily solved. In this work, we shall significant extend the study \cite{LL} in two aspects. First, the study in \cite{LL} is to recover a polygon in the 2D plane, whereas in this study we shall recover a polyhedron in the 3D space. As we shall see, this will create much more difficulties in both theoretical and computational aspects. Second, the study in \cite{LL} mainly concerns the acoustic scattering governed by the scalar Helmholtz system, whereas in the present paper, we shall be concerned with the much more complicated vectorial Maxwell system. The proposed scheme for the recovery of an electromagnetic obstacle follows a similar spirit to \cite{LL} by using the high-frequency asymptotics of the electromagnetic waves as well as the local maximum behavior of the backscattering far-field pattern. However, we would like to emphasize that the extension is highly nontrivial and technical.

The rest of the paper is organized as follows. In Section 2, using the physical optics approximation, we prove the local maximum behavior of the modulus of the electric far-field pattern in the backscattering aperture. In Section 3, we present the recovery scheme. Section 4 is devoted to numerical examples, which illustrate the effectiveness of the proposed recovery scheme. Concluding remarks are given in section 5.

\section{Local Maximum Behavior}

In this section, we consider the local maximum behavior of the modulus of the electric far-field pattern $|\bm{E}^\infty(\hat x)|$ corresponding to a polyhedral obstacle $D$. It is first noted that due the symmetric role of the electric field $\bm{E}$ and the magnetic field $\bm{H}$, we would consider the scattering from a PEC obstacle only. Indeed, it is easily verified by letting $\widetilde{\bm{E}}=-\bm{H}$ and $\widetilde{\bm{H}}=\bm{E}$ that $\curl \widetilde{\bm{E}}-ik\widetilde{\bm{H}}=0$ and $\curl \widetilde{\bm{H}}+ik\widetilde{\bm{E}}=0$. Hence, if $D$ is a PMC obstacle with respect to $(\bm{E}, \bm{H})$, then it is a PEC obstacle with respect to $(\widetilde{\bm{E}},\widetilde{\bm{H}})$. Therefore, we focus on the PEC case in what follows and all our subsequent results derived for the scattering from a PEC obstacle equally hold for the scattering from a PMC obstacle.

Throughout the rest of this section, we let $\bm{p}\in\mathbb{R}^3$, $k\in\mathbb{R}_+$ and $\bm{d}\in\Sphere^{2}$ be fixed.
Let $D$ be a convex polyhedron in $\mathbb{R}^3$, such that
\begin{equation}\label{eq:boundary}
\partial D=\bigcup_{j=1}^m C_j,
\end{equation}
where each $C_j$ represents an open face of $\partial D$. Let $\nu(\bm{x})\in\Sphere^{2}$, $\bm{x}\in \partial D$ denote the unit normal vector to $\partial D$ pointing to the exterior of $D$, and we set
\begin{equation}\label{eq:normal}
\bm{\nu}_j:=\bm{\nu}(\bm{x})\ \ \mbox{when}\ \ \bm{x}\in C_j,\ \ j=1,2,\ldots, m.
\end{equation}
Obviously, $\bm{\nu}_j$ is a constant unit vector. Define
\[
\partial D^{+}:=\{\bm{x}\in\partial D:\ \nu(\bm{x})\cdot \bm{d}\geq 0\}\quad\mbox{and}\quad \partial D^{-}:=\{\bm{x}\in\partial D: \ \nu(\bm{x})\cdot \bm{d}< 0\}
\]
to be, respectively, the back-view and front-view of $\partial D$ with respect to the incident direction $\bm{d}$. A face lying in the front-view (resp. back-view) of $\partial D$ will be referred to as a front-face (resp. back-face) with respect to the incident direction $\bm{d}$. Henceforth, $D$ shall be referred to as a polyhedral obstacle.

Let $h_j>0$, $j=0,\cdots,4$ be five fixed a priori constants. A polyhedral obstacle $D$ is said to be {\em admissible} if there hold
  \begin{numcases}{}
   h_0 \leq |D| \leq h_1, \label{eq:cond1} \\
   \displaystyle{\min_{1 \leq \alpha,\alpha' \leq \beta, \alpha \neq \alpha'} |\bm{\nu}_\alpha \times \bm{\nu}_{\alpha'} | \geq h_2}\ \mbox{for each}\ \partial D^- = \bigcup_{\alpha=1}^\beta C_\alpha, \label{eq:cond2} \\
   \displaystyle{\min_{1\leq j\leq m} |C_j| \geq h_3}, \label{eq:cond3} \\
   \displaystyle{\max_{1\leq j\leq m} |\partial C_j| \leq h_4}. \label{eq:cond4}
  \end{numcases}
In \eqref{eq:cond1}, \eqref{eq:cond3} and \eqref{eq:cond4} we denote by $|\Omega|$ the volume, the area and the perimeter of $\Omega$ respetively. Condition \eqref{eq:cond1} means that the polyhedron is of regular size (with respect to the wavelength). Generically speaking, condition \eqref{eq:cond2} excludes the case when two faces in a back-view of $D$ are nearly parallel to each other. Condition \eqref{eq:cond3} means no face of $D$ can be too small and condition \eqref{eq:cond4} means the faces of $D$ are mildly ``round''. It is worth to point out that only the first three conditions are required in the two-dimensional case.

Let $h_5>0$ be another fixed a priori constant. A face $C_j \in \partial D^-$ is said to be {\em significant} with respect to the incident direction $\bm{d}$ if
\begin{align}
  |\bm{d} \cdot \bm{\nu}_j| \geq h_5. \label{eq:cond5}
\end{align}
Intuitively speaking, condition \eqref{eq:cond5} means $C_j$ is not too parallel to $\bm{d}$ so that the scattered field contributed by $C_j$ is significant. 

For the subsequent use, we let
\begin{equation}\label{eq:fundamental}
\Phi(\bm{x},\bm{y}):=\frac{1}{4\pi}\frac{e^{ik|\bm{x}-\bm{y}|}}{|\bm{x}-\bm{y}|},\quad \bm{x}, \bm{y}\in\mathbb{R}^3,\ \ \bm{x}\neq \bm{y},
\end{equation}
which satisfies $-(\Delta_{\bm{x}}+k^2)\Phi(\bm{x},\bm{y})=\delta(\bm{x}-\bm{y})$.
Let $\psi(\bm{x})\in \mathcal{C}(\partial D)^3$, $\bm{x}\in\partial D$, and define
\begin{equation}\label{eq:vp}
\mathscr{V}(\bm{x}):=\int_{\partial D} \Phi(\bm{x},\bm{y}) \psi(\bm{y})\, {\rm d}s_{\bm{y}},\quad \bm{x} \in\mathbb{R}^3 \setminus \partial D.
\end{equation}
Here $\mathscr{V}$ is called a vector potential with the density $\psi$. The following jump relation is known (cf. \cite{CK,Ned}),
\begin{equation}\label{eq:jump}
\bm{\nu}(\bm{x}) \times \curl \mathscr{V}^\pm(\bm{x}) = \bm{\nu}(\bm{x}) \times \curl \int_{\partial D} \Phi(\bm{x},\bm{y}) \psi(\bm{y}) \, {\rm d}s_{\bm{y}} \pm \frac 1 2 \psi(\bm{x})
\end{equation}
for $\bm{x}\in\partial D$ where $\bm{\nu}(\bm{x})\times \curl \mathscr{V}^\pm(\bm{x})$ denotes the limit of $\bm{\nu}(\bm{x})\times \curl \mathscr{V}(\bm{x})$ as $\bm{x}$ approaches $\partial D$, respectively, from the inside and outside of $D$, and the boundary integral is understood as an improper integral.

It is noted that both $\bm{E}$ and $\bm{H}$ to the Maxwell equations \eqref{eq:Maxwell2} satisfy the vectorial Helmholtz equation, i.e. 
\begin{align*}
  (\Delta+k^2) \bm{E}=(\Delta+k^2) \bm{H}=0.
\end{align*}

By the local boundary regularity estimate (cf. \cite{Mcl}), we know that both the total wave fields $\bm{E}$ and $\bm{H}$ to the scattering problem \eqref{eq:Maxwell2}--\eqref{eq:radiation} are continuous up to the boundary.

The following lemma on the representation of the EM wave fields can be found in \cite[Theorem 6.22]{CK}.

\begin{lem}\label{lem:rep}
For EM wave fields of the scattering problem \eqref{eq:Maxwell2}--\eqref{eq:radiation} due to a PEC obstacle $D$, we have
\begin{equation}\label{eq:rep1}
\begin{split}
\bm{E}(\bm{x}) &= \bm{E}^i(\bm{x})-\frac{1}{ik} \curl \curl \int_{\partial D} \Phi(\bm{x}, \bm{y}) \left[ \bm{\nu}(\bm{y}) \times \bm{H}(\bm{y}) \right] \, {\rm d}s_{\bm{y}}, \\[1ex]
\bm{H}(\bm{x}) &= \bm{H}^i(\bm{x})+ \curl \int_{\partial D} \Phi(\bm{x},\bm{y}) \left[ \bm{\nu}(\bm{y})\times \bm{H}(\bm{y}) \right] \, {\rm d}s_{\bm{y}},
\end{split}
\end{equation}
for $\bm{x}\in\mathbb{R}^3\setminus\overline{D}$. The corresponding far-field patterns are given by
\begin{equation}\label{eq:repfar}
\begin{split}
\bm{E}^\infty(\hat{\bm{x}})&= \frac{ik}{4\pi}\hat{\bm{x}}\times\int_{\partial D} e^{-ik\hat{\bm{x}}\cdot \bm{y}} \left[ \bm{\nu}(\bm{y})\times \bm{H}(\bm{y})\times\hat{\bm{x}} \right] \, {\rm d}s_{\bm{y}}, \\[1ex]
\bm{H}^\infty(\hat{\bm{x}})&= \frac{ik}{4\pi}\hat{\bm{x}}\times\int_{\partial D} e^{-ik\hat{\bm{x}}\cdot \bm{y}} \left[ \bm{\nu}(\bm{y})\times \bm{H}(\bm{y}) \right] \, {\rm d}s_{\bm{y}},
\end{split}
\end{equation}
for $\hat{\bm{x}}\in\Sphere^2$.
\end{lem}

Using the second equation in \eqref{eq:rep1}, we have
\begin{equation}\label{eq:pp1}
\begin{split}
\bm{\nu}_j\times \bm{H}(\bm{x})&= \bm{\nu}_j\times \bm{H}^i(\bm{x})+\bm{\nu}_j\times \curl \int_{\partial D} \Phi(\bm{x},\bm{y}) \left[ \bm{\nu}(\bm{y})\times \bm{H}(\bm{y}) \right] \, {\rm d}s_{\bm{y}} \\[1ex]
&= \bm{\nu}_j\times \bm{H}^i(\bm{x})+\bm{\nu}_j\times \curl \sum_{l=1}^m\int_{{C}_l} \Phi(\bm{x}, \bm{y}) \left[ \bm{\nu}_l\times \bm{H}(\bm{y}) \right] \, {\rm d}s_{\bm{y}},
\end{split}
\end{equation}
for $\bm{x} \in \mathbb{R}^3 \setminus \overline{D}$ and $j=1,2,\ldots, m$. Set
\begin{equation}\label{eq:pp2}
\bm{H}_j(\bm{x})=\bm{H}(\bm{x})\quad\mbox{and}\quad \bm{H}_j^i(\bm{x})=\bm{H}^i(\bm{x})\quad \mbox{for}\ \ \bm{x}\in {C}_j,\ \ j=1,2,\ldots,m.
\end{equation}
By letting $\bm{x} \rightarrow \partial D^+$ in \eqref{eq:pp1}, and using the jump relation \eqref{eq:jump}, we have
\begin{equation}\label{eq:pp3}
\frac 1 2 \bm{\nu}_j\times \bm{H}_j(\bm{x})=\bm{\nu}_j\times \bm{H}_j^i(\bm{x})+\bm{\nu}_j\times \curl \sum_{l=1}^m\int_{C_l} \Phi(\bm{x}, \bm{y}) \left[ \bm{\nu}_l\times \bm{H}_l(\bm{y}) \right] \, {\rm d}s_{\bm{y}}
\end{equation}
for $\bm{x}\in C_j$ and $j=1,2,\ldots, m$. Noting that
\begin{equation}\label{eq:pp5}
\begin{split}
& \bm{\nu}_j\times  \curl \int_{C_j} \Phi(\bm{x},\bm{y}) \left[ \bm{\nu}_j\times \bm{H}_j(\bm{y}) \right] \, {\rm d}s_{\bm{y}}\\
&= - \bm{\nu}_j \times \int_{C_j} \left[ \bm{\nu}_j\times \bm{H}_j(\bm{y}) \right] \times {\rm grad}_{\bm{x}} \, \Phi(\bm{x}, \bm{y})\, {\rm d}s_{\bm{y}}=0
\end{split}
\end{equation}
we can rewrite \eqref{eq:pp3} as
\begin{equation}\label{eq:pp6}
\frac 1 2 \bm{\nu}_j\times \bm{H}_j(\bm{x})=\bm{\nu}_j\times \bm{H}_j^i(\bm{x})+\sum_{l=1, l\neq j}^m\bm{\nu}_j\times \curl \int_{C_l}  \Phi(\bm{x}, \bm{y}) \left[ \bm{\nu}_l\times \bm{H}_l(\bm{y}) \right] \, {\rm d}s_{\bm{y}}
\end{equation}
for $\bm{x}\in C_j$ and $j=1,2,\ldots, m$. Since $\Phi(\bm{x}, \bm{y})$ is a real analytic function in $\bm{x}$ for $\bm{x}\in C_j$ and $\bm{y}\in C_l$ with $l\neq j$, one immediately sees from \eqref{eq:pp6} that $\bm{H}(\bm{x})$ is real analytic for $\bm{x}\in C_j$.

Next, we discuss the high-frequency asymptotics or the physical optics approximation of the electromagnetic plane wave scattering from a convex PEC polyhedron, which forms the basis for the current study. It states that the total electric or magnetic wave fields near the boundary of the obstacle are composed of two parts: the direct contribution from the incident wave and the reflected wave where they are present, and the contribution due to the diffraction from the corners and edges of the obstacle. The first contribution is the so-called physical optics approximation. Let $D$ be an admissible polyhedral PEC obstacle and let $C_j\subset\partial D^-$, $1\leq j\leq m$, be a front face of the obstacle. Here and in what follows, we let $C_j$ be parameterized as
\begin{equation}\label{eq:para1}
\langle \bm{\nu}_j, \bm{x}\rangle =l_j,
\end{equation}
where $l_j$ denote the distance from the origin to the plane in $\mathbb{R}^3$ containing $C_j$. Let $C_j^0$ denote the affine plane of $C_j$, i.e., $\langle \bm{\nu}_j, \bm{x}\rangle =0$ for $\bm{x}\in C_j^0$ and let $\mathcal{R}_{C_j^0}$ denote the usual Euclidean reflection in $\mathbb{R}^3$ with respect to $C_j^0$. Now, we consider the scattering near the face $C_j$ of the PEC obstacle $D$ due to an incident plane wave $(\bm{E}^i, \bm{H}^i)$ in \eqref{eq:planewave}. Let $\bm{x}_0^j\in C_j$ be any fixed point and set
\begin{equation}\label{eq:ref1}
\mathcal{H}(\bm{x}):=\nabla\times (\mathcal{R}_{C_j^0} \bm{p}) e^{ik(\bm{x}-\bm{x}_0)\cdot (\mathcal{R}_{C_j^0} \bm{d})} \, e^{ik\bm{x}_0\cdot \bm{d}}.
\end{equation}
It is straightforward to verify, though with a bit tedious calculations, that $\mathcal{H}(\bm{x})$ and
\begin{equation}\label{eq:ref2}
\mathcal{E}(\bm{x}):=\frac{i}{k}\nabla\times \mathcal{H}(\bm{x}),
\end{equation}
are entire solutions to the Maxwell equations \eqref{eq:Maxwell2}. Moreover, $\bm{\nu}_j\times (\bm{E}+\mathcal{E})(\bm{x})=0$ on $C_j$. In fact $\mathcal{E}$ and $\mathcal{H}$ are, respectively, the locally reflected wave fields of $\bm{{E}^i}$ and $\bm{H}^i$ with respect to $C_j$. Therefore, using the physical optics approximation, one would have
\begin{equation}\label{eq:poa1}
\bm{\nu}_j\times \bm{H}(\bm{x})=\bm{\nu}_j\times (\bm{H}^i+\bm{H}^s)(\bm{x})\approx\bm{\nu}_j\times (\bm{H}^i+\mathcal{H})(\bm{x})=2\bm{\nu}_j\times \bm{H}^i(\bm{x})\ \mbox{for}\ \bm{x}\in C_j.
\end{equation}
A rigorous mathematical justification of the above physical optics approximation is fraught with significant difficulties. Indeed, most of the available results in the literature mainly concern the scalar wave scattering governed by the Helmholtz equation; see \cite{CWL,HLM,LaxPhi,Maj,MelTay}. However, even for the scalar case, the rigorous justification of the physical optics approximation is still not fully understood; see \cite{CWL} for an excellent account of the existing theoretical and computational progresses in the literature. In the present work, we focus on the study of the corresponding inverse scattering problem by assuming that the physical optics approximation holds true. It is interesting to note that our theoretical and numerical results clearly validate such approximation.

Summarizing the above discussion, we have

\begin{lem}\label{lem:h1}
Let $\bm{E}$ and $\bm{H}$ be the total wave fields of the scattering problem \eqref{eq:Maxwell2}--\eqref{eq:radiation} due to an admissible polyhedral PEC obstacle $D$. Using the physical optics approximation, one has
\begin{equation}\label{eq:hh1}
\bm{\nu}(\bm{x})\times \bm{H}(\bm{x})\approx 
\begin{cases}
2\bm{\nu}_j(\bm{x})\times \bm{H}^i(\bm{x}),\ \ & \bm{x}\in C_j\subset \partial D^-,\ \ 1\leq j\leq m,\\[1ex]
\qquad 0,\ \ & \bm{x}\in C_{j'}\subset\partial D^+,\ \ 1\leq j'\leq m.
\end{cases}
\end{equation}
\end{lem}

We proceed to derive the local maximum behavior of $|\bm{E}^\infty(\hat{\bm{x}})|$ for our study of the inverse scattering problem. Let
\[
\Sphere^{2}_+:=\{\hat{\bm{x}}\in\Sphere^{2}: \ \hat{\bm{x}}\cdot \bm{d}\geq 0\}\quad\mbox{and}\quad \Sphere^{2}_-:=\{\hat{\bm{x}}\in\Sphere^{2}: \ \hat{\bm{x}}\cdot \bm{d}<0\}
\]
denote, respectively, the forward-scattering and backscattering apertures. Let $C_j\subset\partial D^-$, $1\leq j\leq m$, be a front-face of $\partial D$, and $\bm{\nu}_j\in \Sphere^{2}_-$ denote its unit normal vector pointing to the exterior of $D$.  Define
\begin{equation}\label{eq:parallel}
  \hat{\bm{x}}_j = \mathcal{R}_{C_j^0} \bm{d} = \bm{d} - (2 \bm{d} \cdot \bm{\nu}_j) \bm{\nu}_j
\end{equation}
to be the critical observation direction with respect to $\bm{d}$ and $\bm{\nu}_j$. It is directly calculated that one has
\begin{equation}\label{eq:cd3}
  \bm{\nu}_j=\frac{\hat{\bm{x}}_j-\bm{d}}{\sqrt{2(1-\hat{\bm{x}}_j\cdot \bm{d})}}.
\end{equation}

\begin{defn}
  Let $A: \Sphere^2 \to \RR^+$ be a continuous function. A point $\hat{\bm{z}} \in \Sphere^2$ is said to be an {\em approximate local maximum} of $A$ if there exists a neighborhood $V \subset \Sphere^2$ of $\hat{\bm{x}}$ such that
  \begin{align*}
    A(\hat{\bm{x}})=A_0(\hat{\bm{x}})+A_1(\hat{\bm{x}}), \quad \hat{\bm{x}} \in V,
  \end{align*}
  and $\hat{\bm{z}}$ is the usual local maximum of $A_0$ in $V$ with $\max_{\hat{\bm{x}} \in V} A_0(\hat{\bm{x}}) \gg \max_{\hat{\bm{x}} \in V}  A_1(\hat{\bm{x}})$. 
\end{defn}

\begin{thm}\label{thm:main}
Let $D$ be an admissible polyhedral PEC obstacle with respect to the incident plane wave $(\bm{E}^i,\bm{H}^i)$ in \eqref{eq:planewave}. Suppose that $C_j\subset\partial D^-$ is a front face of the obstacle, and $\bm{\nu}_j$ is the unit normal vector to $C_j$ pointing to the exterior of $D$, $1\leq j\leq m$. Let $\hat{\bm{x}}_j\in\Sphere^{2}$ be the critical observation direction with respect to $\bm{d}$ and $\bm{\nu}_j$. Under the physical optics approximation of Lemma~\ref{lem:h1}, $\hat{\bm{x}}_j$ is an approximate local maximum point of $|\bm{E}^\infty(\hat{\bm{x}})|$ as well as $|\bm{H}^\infty(\hat{\bm{x}})|$, and the maximal value is given by 
\begin{align}
  \label{eq:13}
  |\bm{E}^\infty(\hat{\bm{x}}_j)| \approx |\bm{H}^\infty(\hat{\bm{x}}_j)| \approx \frac{|C_j|}{\lambda} |\bm{d} \cdot \bm{\nu}_j|,
\end{align}
where $\lambda=2\pi/k$ denotes the wavelength. Moreover $\hat{\bm{x}}=\bm{d}$ is also a local maximum of $|\bm{E}^\infty(\hat{\bm{x}})|$ as well as $|\bm{H}^\infty(\hat{\bm{x}})|$,and the maximal value is given by
\begin{align*}
  |\bm{E}^\infty(\bm{d})| \approx |\bm{H}^\infty(\bm{d})| \approx \sum_{\alpha=1}^\beta \frac{|C_\alpha|}{\lambda} |\bm{d} \cdot \bm{\nu}_\alpha|,
\end{align*}
where the sum is taken such that
\begin{align*}
  \partial D^- = \bigcup_{\alpha=1}^\beta C_\alpha
\end{align*}
with $\alpha \in \{1,\cdots,m\}$.
\end{thm}

\begin{proof}
We first prove that $\hat{\bm{x}}_j$ is a local maximum of $|\bm{H}^\infty(\hat{\bm{x}})|$. By using the integral representation \eqref{eq:repfar} and the physical optics approximation \eqref{eq:hh1}, one has
\begin{equation}\label{eq:pc1}
\bm{H}^\infty(\hat{\bm{x}})\approx \frac{ik}{2\pi}\hat{\bm{x}}\times\int_{\partial D^-} e^{-ik\hat{\bm{x}}\cdot \bm{y}} \left[ \bm{\nu}(\bm{y})\times \bm{H}^i(\bm{y}) \right] \, {\rm d}s_{\bm{y}},
\end{equation}
Using the form of $\bm{H}^i$ in \eqref{eq:planewave}, one further has by direct calculations that
\begin{align}
\bm{H}^\infty(\hat{\bm{x}}) & \approx \frac{ik}{2\pi} \sum_{\alpha=1}^\beta \hat{\bm{x}}\times \left[ \bm{\nu}_\alpha \times (\bm{d}\times\bm{p}) \right] \int_{C_\alpha} e^{ik(\bm{d}-\hat{\bm{x}})\cdot\bm{y}} \, {\rm d} s_{\bm{y}} \label{eq:16} \\[1ex]
&=\frac{ik}{2\pi}\left[\bm{A}_1(\hat{\bm{x}})+\bm{A}_2(\hat{\bm{x}}) \right],\quad \hat{\bm{x}}\in\Sphere^2 \label{eq:pc2},
\end{align}
where
\begin{align}
\bm{A}_1(\hat{\bm{x}})& := \hat{\bm{x}} \times \left[ \bm{\nu}_j \times (\bm{d}\times\bm{p}) \right] \int_{C_j} e^{ik(\bm{d}-\hat{\bm{x}})\cdot\bm{y}} \, {\rm d}s_{\bm{y}} \\[1ex]
\bm{A}_2(\hat{\bm{x}}) & := \sum_{\alpha=1,\alpha \neq j}^\beta \hat{\bm{x}}\times \left[ \bm{\nu}_\alpha \times (\bm{d}\times\bm{p}) \right] \int_{C_\alpha} e^{ik(\bm{d}-\hat{\bm{x}})\cdot\bm{y}} \, {\rm d} s_{\bm{y}} \label{eq:5}
\end{align}
Clearly, in order to consider the local maximum behavior of $|\bm{H}^\infty(\hat{\bm{x}})|$, it suffices to consider the local maximum behavior of $|\bm{A}_1(\hat{\bm{x}})+\bm{A}_2(\hat{\bm{x}})|$.

Let $\Sigma_j\subset\Sphere^2$ be a small neighborhood of $\hat{\bm{x}}_j$. Fix an $\hat{\bm{x}} \in \Sigma_j$ and let
\begin{align*}
  \bm{\tau}_\alpha = \bm{\nu}_\alpha \times (\bm{d} - \hat{\bm{x}}), \quad \bm{F} = \frac{1}{ik} e^{ik(\bm{d}-\hat{\bm{x}}) \cdot \bm{y}} \bm{\tau}_\alpha.
\end{align*}
By direct calculation we have from \eqref{eq:parallel} and condition \eqref{eq:cond2} that
\begin{align}
  \label{eq:2}
  \left| \int_{C_\alpha} {\rm curl}_{\bm{y}} \, \bm{F} \cdot {\rm d}s_{\bm{y}} \right| = \left| \bm{\tau}_\alpha \right|^2 \left| \int_{C_\alpha} e^{ik(\bm{d}-\hat{\bm{x}}) \cdot \bm{y}} \, {\rm d}s_{\bm{y}} \right| 
  \gtrsim \left| \int_{C_\alpha} e^{ik(\bm{d}-\hat{\bm{x}}) \cdot \bm{y}} \, {\rm d}s_{\bm{y}} \right|, \quad \alpha \neq j,
\end{align}
Henceforth the expression $A \gtrsim B$ (resp. $A \lesssim B$) means $A \geq c B$ (resp. $A \leq c B$) for some positive constant $c$ depending only on the a priori constants $h_i,i=1,\cdots,5$ defined in conditions \eqref{eq:cond1}--\eqref{eq:cond5}. On the other hand we have from Stoke's theorem and condition \eqref{eq:cond4} that
\begin{align}
  \label{eq:3}
  \left| \int_{C_\alpha} {\rm curl}_{\bm{y}} \, \bm{F} \cdot {\rm d}s_{\bm{y}} \right| =
  \left| \int_{\partial C_\alpha} \bm{F} \cdot {\rm d}l_{\bm{y}} \right| \lesssim \frac{1}{k},
\end{align}
Combining \eqref{eq:3} and \eqref{eq:2}, we arrive at
\begin{align}
  \label{eq:4}
  \left| \int_{C_\alpha} e^{ik(\bm{d}-\hat{\bm{x}}) \cdot \bm{y}} \, {\rm d}s_{\bm{y}} \right| \lesssim \frac{1}{k}, \quad \alpha \neq j,
\end{align}
Plugging \eqref{eq:4} into \eqref{eq:5}, we have for $k$ sufficiently large
\begin{equation}
  \label{eq:12}
|\bm{A}_2(\hat{\bm{x}})|\ll 1\quad \mbox{for}\ \hat{\bm{x}} \in \Sigma_j.
\end{equation}

Next, we evaluate $\bm{A}_1(\hat{\bm{x}})$ for $\hat{\bm{x}} \in \Sigma_j$. Let $\bm{y}_0^j\in C_j$ be any fixed point, and define $\widetilde{C}_j^0:=C_j-\{\bm{y}_0^j\}$. Recall that $C_j^0$ denotes the affine plane corresponding to $C_j$. One clearly has that $\widetilde{C}_j^0\subset C_j^0$. By straightforward calculations, one has that for $\hat{\bm{x}}\in \Sigma_j$
\begin{align}
|\bm{A}_1(\hat{\bm{x}})|&= \left|\hat{\bm{x}} \times \left[ \bm{\nu}_j\times(\bm{d}\times\bm{p}) \right] \int_{C_j} e^{ik(\bm{d}-\hat{\bm{x}})\cdot \bm{y}} \, {\rm d}s_{\bm{y}}\right| \notag \\[1ex]
&= \left| e^{ik(\bm{d}-\hat{\bm{x}})\cdot \bm{y}_0^j} \, \hat{\bm{x}} \times \left[ \bm{\nu}_j\times(\bm{d}\times\bm{p}) \right] \int_{\widetilde{C}_j^0} e^{ik(\bm{d}-\hat{\bm{x}})\cdot \bm{y}} \, {\rm d}s_{\bm{y}} \right| \notag \\[1ex]
&= \left| \hat{\bm{x}} \times \left[ \bm{\nu}_j\times(\bm{d}\times\bm{p}) \right] \int_{\widetilde{C}_j^0} e^{ik(\bm{d}-\hat{\bm{x}})\cdot \bm{y}} \, {\rm d}s_{\bm{y}} \right|. \label{eq:6}
\end{align}
Let $\hat{\bm{x}} = \hat{\bm{x}}_j + \bm{\varepsilon} \in \Sigma_j$, where $\bm{\varepsilon} \in \mathbb{R}^3$ is such that $|\bm{\varepsilon}| \ll 1$. By virtue of \eqref{eq:parallel}, we can write
\begin{align}
  \label{eq:1}
  \hat{\bm{x}} = -(2 \bm{d} \cdot \bm{\nu}_j) \bm{\nu}_j + \bm{d} + \bm{\varepsilon}.
\end{align}
Substituting \eqref{eq:1} into \eqref{eq:6} and noting that $\bm{\nu}_j \cdot \bm{y} = 0$ for $\bm{y} \in \widetilde{C}_j^0$, we obtain
\begin{align}
  \label{eq:7}
  |\bm{A}_1(\hat{\bm{x}})| = |\bm{A}_{11}(\hat{\bm{x}})| + |\bm{A}_{12}(\hat{\bm{x}})|
\end{align}
with 
\begin{align}
  |\bm{A}_{12}(\hat{\bm{x}})| \ll 1 \label{eq:11}
\end{align}
and
\begin{align}
  \label{eq:8}
  |\bm{A}_{11}(\hat{\bm{x}})| = \left| (\bm{d} \cdot \bm{\nu}_j) \int_{\widetilde{C}_j^0} e^{-ik \bm{\varepsilon} \cdot \bm{y}} \, {\rm d}s_{\bm{y}} \right|.
\end{align}
Clearly the maximum of $|\bm{A}_{11}(\hat{\bm{x}})|$ is achieved at $\bm{\varepsilon}=0$, i.e. $\hat{\bm{x}}=\hat{\bm{x}}_j$ with maximal value
\begin{align} \label{eq:14}
  \max_{\hat{\bm{x}} \in \Sigma_j} |\bm{A}_{11}(\hat{\bm{x}})| = |\bm{d} \cdot \bm{\nu}_j| |C_j|.
\end{align}

Under conditions \eqref{eq:cond3} and \eqref{eq:cond5} we have
\begin{align}
  \label{eq:10}
  \max_{\hat{\bm{x}} \in \Sigma_j} |\bm{A}_{11}(\hat{\bm{x}})| \gtrsim 1.
\end{align}

Combining \eqref{eq:12}, \eqref{eq:11} and \eqref{eq:10} readily implies that $\hat{\bm{x}}_j$ is an approximate local maximum of $|\bm{H}^\infty(\hat{\bm{x}})|$. Combining \eqref{eq:pc2}, \eqref{eq:7} and \eqref{eq:14} implies the second part of \eqref{eq:13}. 

Finally, we note by \eqref{eq:relation} that $|\bm{E}^\infty(\hat{\bm{x}})|=|\bm{H}^\infty(\hat{\bm{x}})|$ for each $\hat{\bm{x}} \in \Sphere^2$. Hence $\hat{\bm{x}}_j$ is also an approximate local maximum of $|\bm{E}^\infty(\hat{\bm{x}})|$ and the first part of \eqref{eq:13} holds.

Finally we show $\hat{\bm{x}}=\bm{d}$ is also a local maximum of $|\bm{H}^\infty(\hat{\bm{x}})|$ and $|\bm{E}^\infty(\hat{\bm{x}})|$. Let $\Sigma \subset \Sphere^2$ be a small neighborhood of $\bm{d}$ and $\hat{\bm{x}} = \bm{d} + \bm{\varepsilon} \in \Sigma$ for some $\bm{\varepsilon} \ll 1$. Substituting  $\hat{\bm{x}} = \bm{d} + \bm{\varepsilon}$ into \eqref{eq:16} yields
\begin{align*}
  {\bm H}^\infty(\hat{\bm{x}}) \approx \frac{ik}{2\pi} \left[ \bm{B}_1(\hat{\bm{x}}) + \bm{B}_2(\hat{\bm{x}}) \right]
\end{align*}
where $|\bm{B}_1(\hat{\bm{x}})| \ll 1$ and
\begin{align*}
  \bm{B}_2(\hat{\bm{x}}) = -\sum_{\alpha=1}^\beta (\bm{d} \cdot \bm{\nu}_\alpha) (\bm{d} \times \bm{p}) \int_{C_\alpha} e^{-ik\bm{\varepsilon}\cdot \bm{y}} \, {\rm d}s_{\bm y}
\end{align*}
Clearly the maximum of $|\bm{B}_2(\hat{\bm{x}})|$ is obtained at $\bm{\varepsilon}=0$, i.e $\hat{\bm{x}} = \bm{d}$ with maximal value 
\begin{align*}
|\bm{B}_2(\bm{d})| = \sum_{\alpha=1}^\beta |\bm{d} \cdot \bm{\nu}_\alpha| |C_\alpha|.
\end{align*}
The argument for the local maximum behavior of $|\bm{E}^\infty(\hat{\bm{x}})|$ is similar and omitted.

The proof is complete.
\end{proof}

\begin{rem}
  Theorem \ref{thm:main} tells that the maximal value of $\bm{H}^\infty$ (and $\bm{E}^\infty$) in the incident direction $\bm{d}$ is approximately the sum of the maximal values of $\bm{H}^\infty$ (and $\bm{E}^\infty$) in the critical observation directions. In the sequel numerical experiments, we see $\bm{d}$ is in fact the global maximum as long as $k$ is sufficiently large, but we do not have a mathematical justification.
\end{rem}

\section{Recovery Scheme} \label{sec:scheme}
Based on Theorem \ref{thm:main}, we propose the following scheme for the recovery of the face normals and areas for the polyhedron $D$.

\noindent {\bf Step 1: Recover the face normals and areas}

\begin{enumerate}
\item Choose a set of incident directions $\{\bm{d}_n\}_{n=1}^N$ such that the union of the corresponding front-view faces covers $\partial D$. Choose a wave number $k$ such that $k|D| \gg 1$.
\item For each $n=1,\cdots,N$ send an incident plane wave of the form \eqref{eq:planewave} with incident direction $\bm{d}_n$ and wavenumber $k$, and collect the phaseless far-field data $|\bm{E}^\infty(\hat{\bm{x}})|, \hat{\bm{x}} \in \Sphere^2$.
\item For each $n=1,\cdots,N$ find the local maxima $\hat{\bm{x}}_j$ of $|\bm{E}^\infty(\hat{\bm{x}})|$ and the maximal value $\max_{\hat{\bm{x}} \in \Sphere^2} |\bm{E}^\infty(\hat{\bm{x}}_j^n)|$. 

In practice the data is not measured on every point on the sphere but only a discrete set $\mathcal{T} \subset \Sphere^2$ of grid points. Besides, there always exists measurement noise. Hence we need to find the local maxima of a set of scattered data (with noises) on the unit sphere. To our best knowledge there exists no specially designed algorithm for this task. We propose the following algorithm:
\begin{enumerate}
\item Find the coefficients of the spherical harmonic transform of phaseless far-field data $|\bm{E}^\infty(\hat{\bm{x}})|, \hat{\bm{x}} \in \Sphere^2$ up to a fixed order $n_c \in \mathbb{N}$, i.e.
  \begin{align}
    \label{eq:15}
    c_n^m = \int_{\Sphere^2} |\bm{E}^\infty(\hat{\bm{x}})| Y_n^m(\hat{\bm{x}}) \, {\rm d}s, \quad n=0,\cdots,n_c, m=-n,\cdots,n
  \end{align}
\item Once the Fourier coefficients $c_n^m$ are computed, we define the function
  \begin{align}
    \label{eq:17}
    f(\hat{\bm{x}}) = \Re\,\left( \sum_{n=0}^{n_c} \sum_{m=-n}^n c_n^m Y_n^m(\hat{\bm{x}}) \right), \quad \hat{\bm{x}} \in \Sphere^2
  \end{align}
as a continuous approximation of the measurement data and find all its local maxima efficiently using nonlinear optimization algorithms.
\end{enumerate}
The advantage of the above scheme is two-fold. First, it converts the discrete data into a continuous function so that its local maxima can be found efficiently. Second, the measurement noise can be filtered out by controlling the cut-off frequency $n_c$ and the effect of noise on the result decreases as the sampling rate increases.

Some of the local maxima $\hat{\bm{x}}_j$ does not correspond to critical observation directions. By Theorem \ref{thm:main} the direction $\hat{\bm{x}}=\bm{d}$ is also a local maximum of $|\bm{E}^\infty(\hat{\bm{x}})|$. Hence this direction is exluded from the search for face normals. If we have an a priori estimate of
\begin{align*}
  \sigma=\min \{ |\hat{\bm{x}}_j - \bm{d}_n|: \hat{\bm{x}}_j \mbox{ is a critical observation direction for } \bm{d}_n \},
\end{align*}
then we can also exclude all directions within a distance of $\sigma$ from $\bm{d}_n$. Besides those directions, there may exists other local maxima that are not associated with any critical observation direction.
Numerical experience indicates those local maxima usually attain smaller values than the critical observation directions and the incident direction $\bm{d}_n$. Hence we also delete the local maxima $\hat{\bm{x}}_j^n$ such that such that $|\bm{E}^\infty(\hat{\bm{x}}_j)| < \bm{E}^\infty_{\rm tol}$ for some a priori settled threshold $\bm{E}^\infty_{\rm tol}$. In view of \eqref{eq:13}, this threshold may be chosen as
\begin{align*}
  \bm{E}^\infty_{\rm tol} = \frac{1}{\lambda} \min \{ |C_j| \min | \bm{d}_n \cdot \bm{\nu}_j |: C_j \mbox{ is a front-view face for } \bm{d}_n \}.
\end{align*}

Finally we compute the face normals $\bm{\nu}_j$ from $\hat{\bm{x}}_j$ using \eqref{eq:cd3}.
\item Due to the overlaps in the front-view faces for all incident directions, some of the normals $\bm{\nu}_j$ obtained in the previous step are very close and corresponds to the same face. For a cluster of $\bm{\nu}_j$ that are very close to each other, we choose the one corresponding to the largest $|\bm{E}^\infty(\hat{\bm{x}}_j)|$ as the {\em effective normal vector}. 
\item According to \eqref{eq:13}, we compute the face area $|C_j| \approx \lambda |\bm{E}^\infty(\hat{\bm{x}}_j)| / |\bm{d}_n \cdot \bm{\nu}_j|$ for each of the effective normal vector $\bm{\nu}_j$. 
\end{enumerate}

Let $\bm{\nu}_1,\cdots,\bm{\nu}_k$ and $A_1,\cdots,A_k$ be the face normals and areas found in step 1. The problem is then reduced to the classical Minkowski problem in computational geometry. The following existence and uniqueness result is known (cf. \cite{Klain}).

\begin{prop}
  Let $\bm{\nu}_1,\cdots,\bm{\nu}_k \in \RR^n$ be unit vectors that span $\RR^n$ and $A_1,\cdots,A_k$ be positive scalars. Then there exists a convex polytope having face normals $\bm{\nu}_1,\cdots,\bm{\nu}_k$ and face areas $A_1,\cdots,A_k$ if and only if $\sum_{j=1}^k A_j \bm{\nu}_j = 0$. Moreover the polytope is unique up to translations.
\end{prop}
Since the face normals and areas we found in step 1 are approximations to those true values from the underlying polyhedra, we have $\sum_{j=1}^k A_j \bm{\nu}_j \approx 0$ and it is expected the unique convex polyhedra recovered from those data is a good approximation to the underlying one up to translation. 

To build the polyhedron from the computed face normals and areas, we must find the equation of the planes containing each face. Since the normals of the planes are already computed, it remains to determine the offsets of the plane from a given reference point in the polyhedron. In view of that the polyhedron is uniquely determined up to translations, we simply choose the origin as the reference point. For a given face normal $\bm{\nu}$ and offset $\alpha \geq 0$, denote by $h(\bm{\nu},\alpha)$ the inward half space formed by the corresponding plane, i.e.
\begin{align*}
  h(\bm{\nu},\alpha) = \{ \bm{x} \in \RR^3: \bm{x}\cdot\bm{\nu} \leq \alpha \}.
\end{align*}
For a fixed set of face normals $\bm{V}=(\bm{\nu}_1,\cdots,\bm{\nu}_k)$ and to-be-determined offsets $\bm{\alpha}=(\alpha_1,\cdots,\alpha_k)$, denote by $a_j(\bm{V},\bm{\alpha})$ the face area of the $j$-th facet of the unique polyhedron formed by the intersection of the half spaces $h(\bm{\nu}_j,\alpha_j),j=1, \cdots,k$. We then proceed as follows.

\noindent {\bf Step 2: Recover face offsets and build the polyhedron}

\begin{enumerate}
\item Determine the face offsets $\bm{\alpha}$ by least-square fitting, i.e.
  \begin{align}
    \label{eq:18}
    \bm{\alpha}={\rm argmin}_{\alpha_1 \geq 0, \cdots, \alpha_k \geq 0}\, \sum_{j=1}^k \left| a_j(\bm{V},\bm{\alpha}) - A_j \right|^2.
  \end{align}
\item Reconstruct the polyhedron as the intersection of the half spaces $h(\bm{\nu}_j,\alpha_j),j=1,\cdots,k$.
\end{enumerate}

The last step of the recovery scheme is to determine the location of the polyhedron. The locating scheme we shall use is a special case of that proposed in \cite{LLSS}. We first introduce the space of $L^2$ tangential fields on the unit sphere
\begin{align*}
  T^2(\Sphere^2) = \{\bm{a} \in L^2(\Sphere^2)^3: \bm{a} \cdot \hat{\bm{x}} = 0 \ \mbox{for a.e.}\ \hat{\bm{x}} \in \Sphere^2\}
\end{align*}
endowed with the inner product $\langle \bm{u}, \bm{v} \rangle_{T^2(\Sphere^2)} = \int_{\Sphere^2} \bm{u} \cdot \bar{\bm{v}} \, {\rm d}s$, and the set of vectorial spherical harmonics
\begin{align*}
  U_n^m(\hat{\bm{x}}) = \frac{1}{\sqrt{n(n+1)}} {\rm Grad}\, Y_n^m(\hat{\bm{x}}), \quad V_n^m(\hat{\bm{x}}) = \hat{\bm{x}} \times U_n^m(\hat{\bm{x}}), \quad \hat{\bm{x}} \in \Sphere^2
\end{align*}
for $ n \in \mathbb{N}$ and $m=-n \cdots,n$. Here $Y_n^m$ denotes the usual spherical harmonic of degree $n$ and order $m$, and ${\rm Grad}$ denotes the surface gradient operator on $\Sphere^2$. It is known that the set of sphereical harmonics form a complete basis for the vector space $T^2(\Sphere^2)$.

We next introduce the indictor function
\begin{align}
  I(\bm{z}) = \frac{1}{\|\bm{E}^\infty(\hat{\bm{x}})\|^2_{T^2(\Sphere^2)}} \sum_{|m| \leq 1} 
  & \left( \left| \left\langle \bm{E}^\infty(\hat{\bm{x}}), e^{ik(\bm{d}-\hat{\bm{x}}) \cdot \bm{z}} U_1^m(\hat{\bm{x}}) \right\rangle_{T^2(\Sphere^2)} \right|^2 \right. \notag \\
& \left. + \left| \left\langle \bm{E}^\infty(\hat{\bm{x}}), e^{ik(\bm{d}-\hat{\bm{x}}) \cdot \bm{z}} V_1^m(\hat{\bm{x}}) \right\rangle_{T^2(\Sphere^2)} \right|^2 \right), \quad \bm{z} \in \RR^3.   \label{eq:9}
\end{align}
According to Theorem 2.1 in \cite{LLSS}, we can deduce that a fixed point contained in $D$ is an approximate local minimum of $I(\bm{z})$ if $k |D| \ll 1$. Using this result, we propose the following scheme for the recovery of the location of $D$:

\noindent {\bf Step 3: Recover the location}

\begin{enumerate}
\item Send an incident wave of the form \eqref{eq:planewave} with a fixed polarization $\bm{p}$, propagation direction $\bm{d}$ and wavenumber $k$ such that $k|D| \ll 1$. Collect the far-field data $\bm{E}^\infty(\hat{\bm{x}}), \hat{\bm{x}} \in \Sphere^2$.
\item Determine the location of $D$ as a minimizer of the indicator function \eqref{eq:9} in a prescribed sample region $S$. Intuition and numerical experiments indicate the minimizer is unique if $k|S|$ is sufficiently small. 
\end{enumerate}

\section{Numerical Experiments}
In this section, we conduct numerical experiments to verify the proposed recovery scheme Steps 1--3 in Section \ref{sec:scheme}.

We first consider the simplest polyhedron, i.e. a tetrahedron to demonstrate how the recovery scheme works. Let $D$ be the tetrahedron with vertex coordinates given in Table \ref{tab:tetrahedron-vertex} and face-vertex adjacense relation defined in Table \ref{tab:tetrahedron-faces}, which defines a regular tetrahedron with unit side length and center of gravity at the origin (cf. Figure \ref{fig:tetrahedron}). The tetrahedron has four vertices and four triangular faces. We then translate the tetrahedron so that its center of gravity moves to $(50,50,50)$. Note that the location of the tetrahedron plays no role in step 1--2 of the recovery scheme since they use only the norm of the far-field data, which is translation invariant for plane wave incidence.
\begin{table}[t]
  \centering
  \begin{tabular}{ccccc}
    \toprule
    Vertices & $P_1$ & $P_2$ & $P_3$ & $P_4$ \tabularnewline
                            \midrule
    Coordinates & $\left( \frac{1}{2},0,-\frac{1}{\sqrt{8}} \right)$ & $\left( -\frac{1}{2},0,-\frac{1}{\sqrt{8}} \right)$ & $\left( 0,\frac{1}{2},\frac{1}{\sqrt{8}} \right)$ & $\left( 0,-\frac{1}{2},\frac{1}{\sqrt{8}} \right)$ \tabularnewline
\bottomrule               
  \end{tabular}
\caption{Vertex coordinates of the tetrahedron (centered at origion) to be recovered (cf. Figure \ref{fig:tetrahedron}).}
\label{tab:tetrahedron-vertex}
\end{table}
\begin{table}[t]
  \centering
  \begin{tabular}{ccccc}
    \toprule
    Faces & $C_1$ & $C_2$ & $C_3$ & $C_4$ \tabularnewline
                            \midrule
    Vertices & $P_2P_4P_3$ & $P_1P_3P_4$ & $P_1P_4P_2$ & $P_1P_2P_3$ \tabularnewline
\bottomrule               
  \end{tabular}
\caption{Face-vertex adjacense relation for the tetrahedron to be recovered (cf. Table \ref{tab:tetrahedron-vertex}).}
\label{tab:tetrahedron-faces}
\end{table}
\begin{figure}[t]
  \centering
  \includegraphics[width=0.5\textwidth]{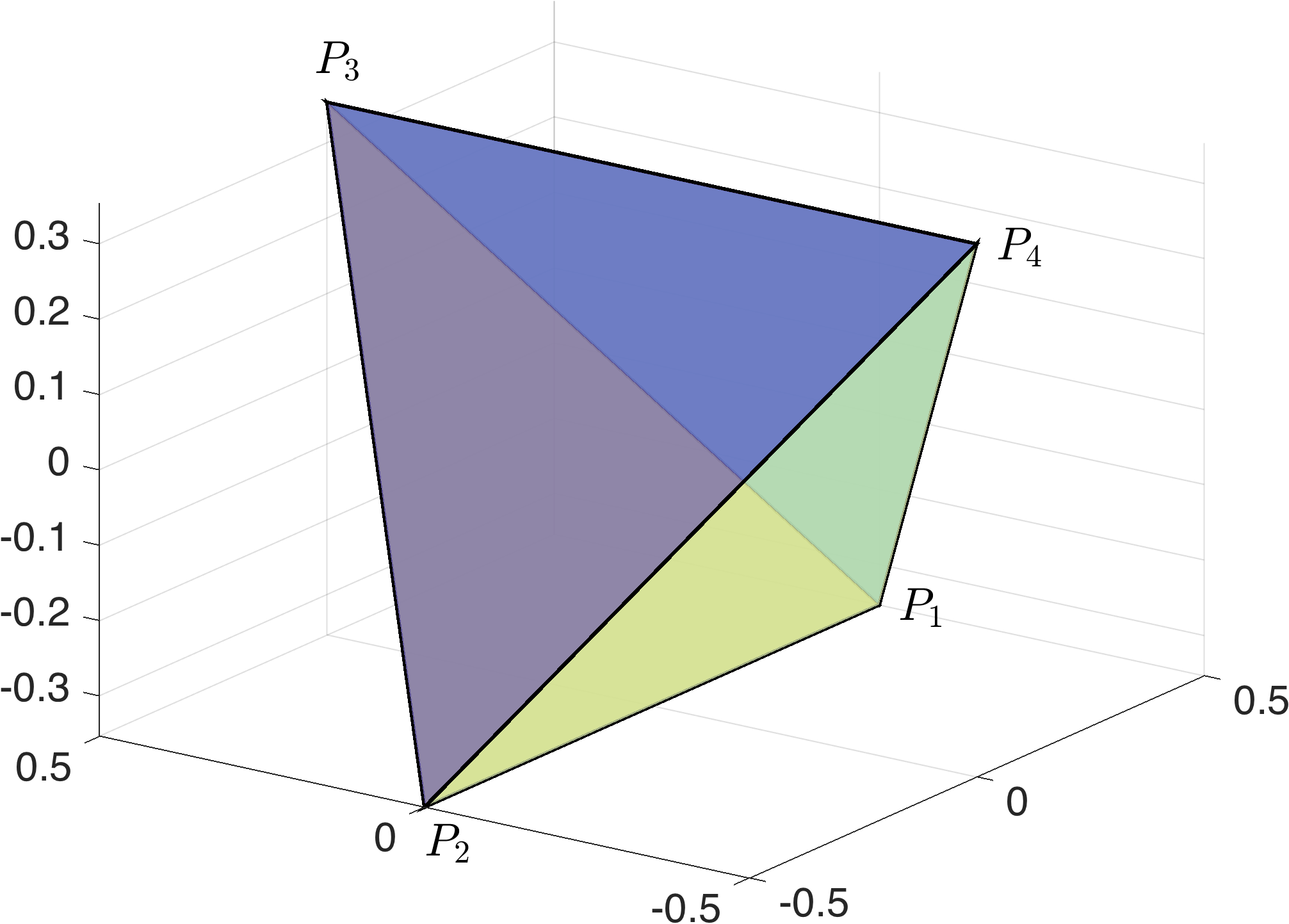}
  \caption{The true tetrahedron to be recovered (cf. Table \ref{tab:tetrahedron-vertex} for vertex coordinates).}
  \label{fig:tetrahedron}
\end{figure}

We first recover the face normals and areas according to the substeps described in Step 1.

\noindent {\bf Step 1: Recover face normals $\bm{\nu}_j$ and face areas $A_j$}
\begin{enumerate}
\item We choose the set of propagation directions $\bm{d}_n$ and polarization vectors $\bm{p}_n$, $n=1,\cdots,6$ listed in Table \ref{tab:incident} for the incident field so that the union of the corresponding front-view faces covers $\partial D$.
  \begin{table}[t]
    \centering
    \begin{tabular}{ccccccc}
      \toprule
      $\bm{d}_n$ & $(1,0,0)$ & $(-1,0,0)$ & $(0,1,0)$ & $(0,-1,0)$ & $(0,0,1)$ & $(0,0,-1)$ \tabularnewline \midrule
     $\bm{p}_n$ & $(0,0,1)$ & $(0,0,1)$ & $(0,0,1)$ & $(0,0,1)$ & $(1,0,0)$ & $(1,0,0)$ \tabularnewline
\bottomrule
    \end{tabular}
    \caption{Propagation directions and polarization vectors for the incident field}
    \label{tab:incident}
  \end{table}
Set the wavelength to be $\lambda=0.5$ (wavenumber $k=4\pi$).
\item Apply the distmesh program \cite{PerssonStrang04} to generate a uniform grid $\mathcal{T}$ consisting of $7518$ grid points on the unit sphere. It is simpler to use a mesh generated by a uniform division of the polar and azimuth angles, but the resulting mesh is less uniform since grid points are clustered around the two poles. Then we send each incident field listed in Table \ref{tab:incident} and collect the phaseless far-field data $\left| \bm{E}^\infty(\hat{\bm{x}}_j) \right|, \hat{\bm{x}}_j \in \mathcal{T}$. The sythetic measurement data in this paper is obtained by solving the direct scattering problem using the finite element method (implemented in COMSOL Multiphysics).
\item Find the local maxima of $\left| \bm{E}^\infty(\hat{\bm{x}}) \right|, \hat{\bm{x}} \in \Sphere^2$. To give an intuitive idea about the distribution of local maxima, we present in Figure \ref{fig:normEfar} a 3D polar plot of $\left| \bm{E}^\infty(\hat{\bm{x}}) \right|, \hat{\bm{x}} \in \Sphere^2$ corresponding to the incident direction $\bm{d}_1=(1,0,0)$. 
  \begin{figure}[t]
    \centering
    \includegraphics[width=0.5\textwidth]{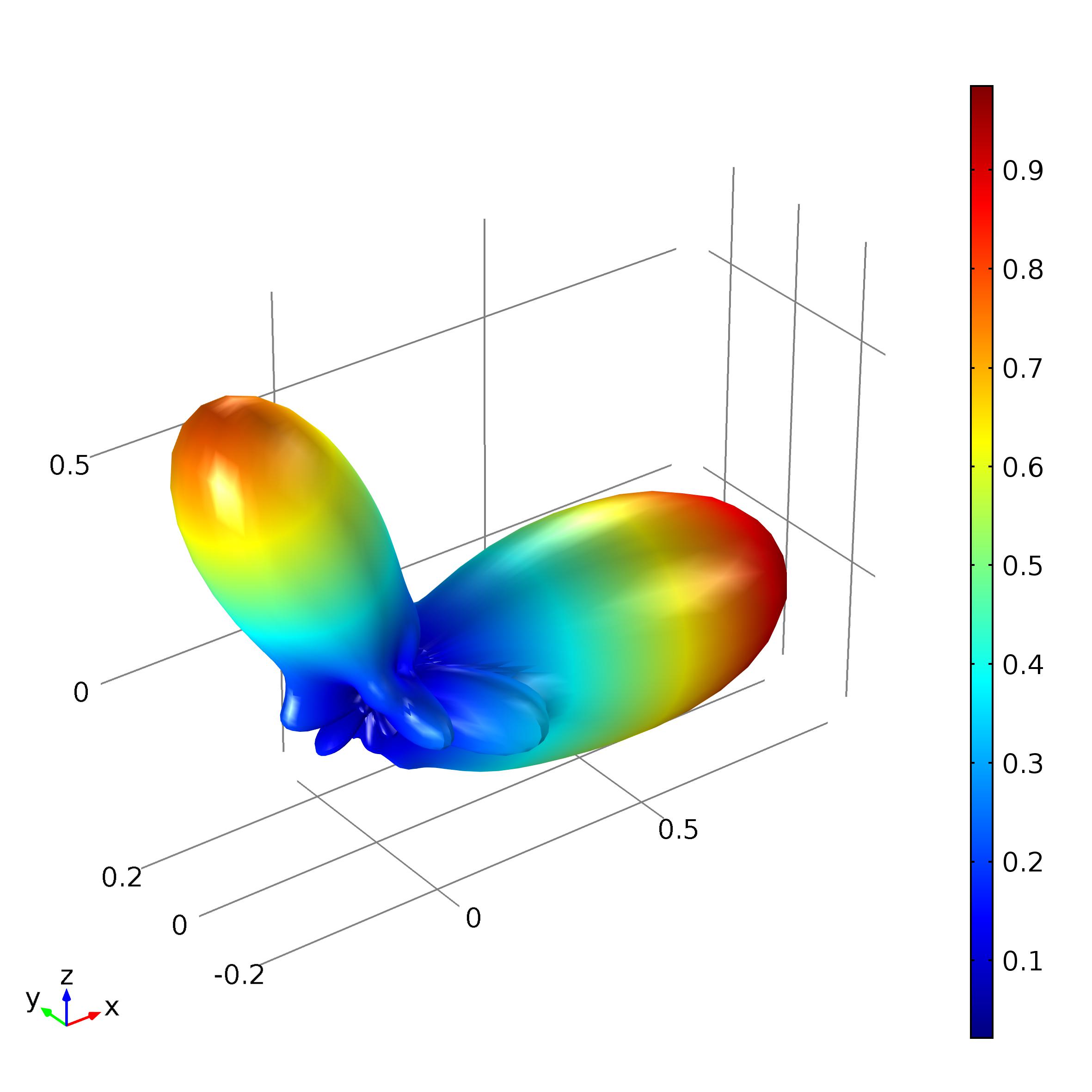}
    \caption{3D polar plot of $\left| \bm{E}^\infty(\hat{\bm{x}}) \right|, \hat{\bm{x}} \in \Sphere^2$ for the tetrahedron and incident direction $\bm{d}_1$.}
    \label{fig:normEfar}
  \end{figure}
Clearly seen in the plot are two major local maxima, one in the incident direction $\bm{d}$ and another in the critical observation direction $\hat{\bm{x}}_1$ corresponds to the only front-view face $C_1$. Besides the two larger maxima, there also exist other local maixma of smaller values. These minor local maxima are not associated to any critical observation directions and hence not considered as critical observation directions. Numerical experience tells the contrast between the major and minor local maxima increases as the wavenumber increases. In view of \eqref{eq:13}, it is possible to obtain a threshold $\bm{E}^\infty_{\rm tol}$ automatically using the a priori parameters so that all local maxima $\hat{\bm{x}}_j$ with $\bm{E}^\infty(\hat{\bm{x}}_j) < \bm{E}^\infty_{\rm tol}$ can be excluded from the critical observation directions. In this particular experiment we take $\bm{E}^\infty_{\rm tol} = 0.5$. 

To find the local maxima, we proceed to Steps (3a) and (3b) as described in section \ref{sec:scheme}.

\begin{enumerate}
\item Choose a cut-off frequency $n_c=10$ and find the Fourier coefficients $c_n^m, n=0,\cdots,n_c, m=-n,\cdots,n$. We approximate the integral in \eqref{eq:15} using the grid values $|\bm{E}^\infty(\hat{\bm{x}}_l)| Y_n^m(\hat{\bm{x}}_l), \hat{\bm{x}}_l \in \mathcal{T}$ on each triangle in the mesh $\mathcal{T}$ and compute $c_n^m$ as
\begin{align}
  \label{eq:19}
  c_n^m \approx \sum_{T \in \mathcal{T}} \frac{1}{3} \sum_{i=1}^3 |\bm{E}^\infty(\hat{\bm{x}}_i)| Y_n^m(\hat{\bm{x}}_i) |C_i|, 
\end{align}
where the first sum runs through all triangles $T$ in the triangulation $\mathcal{T}$ and and the second sum runs through the three vertices $\hat{\bm{x}}_j$ of $T$.
\item Define the objective function $f(\theta,\phi)=f(\hat{\bm{x}})$ as in \eqref{eq:17} and find its local maxima using initial guesses at the grid points on a $5 \times 11$ uniform mesh of $[0,\pi] \times [0,2\pi]$. Delete those maxima $\hat{\bm{x}}_j$ such that $|\bm{E}^\infty(\hat{\bm{x}}_j)| < \bm{E}^\infty_{\rm tol}$.
\end{enumerate}
\begin{rem}
 Steps 3(a) and 3(b) are the most time-consuming part of the scheme and may be accelerated by using the fast spherical harmonic transform algorithms (cf. \cite{KunisPotts03}).
\end{rem}
Once the critical observation directions $\hat{\bm{x}}_j$ are obtained, we compute the corresponding face normals $\bm{\nu}_j$ using \eqref{eq:cd3}. Table \ref{tab:normal-vectors} lists the critical observation directions and correponding maximal values for each incident direction $\bm{d}_n,n=1,\cdots,6$.
\begin{table}[t]
  \centering
  \begin{tabular}{ccc}
    \toprule
    $n$ & $\bm{\nu}_j$ & $|\bm{E}^\infty(\hat{\bm{x}}_j)|$ \tabularnewline
\midrule
                         $1$ & $(-0.85, 0.00, +0.53)$ & 0.80 \tabularnewline
                         $2$ & $(+0.85, 0.00, +0.53)$ & 0.80 \tabularnewline
                         $3$ & $(0.00, -0.85, -0.53)$ & 0.80 \tabularnewline
                         $4$ & $(0.00, +0.85, -0.53)$ & 0.80 \tabularnewline
                         $5$ & $(0.00, +0.75, -0.66)$ & 0.63 \tabularnewline
                         $5$ & $(0.00, -0.75, -0.66)$ & 0.63 \tabularnewline
                         $6$ & $(-0.82, 0.00, +0.57)$ & 0.59 \tabularnewline
                         $6$ & $(+0.82, 0.00, +0.57)$ & 0.59 \tabularnewline 
\bottomrule
  \end{tabular}
  \caption{Recoverd face normals $\bm{\nu}_j$ and corresponding maximal values $|\bm{E}^\infty(\hat{\bm{x}}_j)$ for each incident direction $\bm{d}_n, n=1,\cdots,6$}
  \label{tab:normal-vectors}
\end{table}
\item Due to the overlapping of front-view faces for different incident field, some of the face normals in Table \ref{tab:normal-vectors} are from the same face. For example, the face $C_1$ is a front-view face for both incident direction $\bm{d}_1$ and $\bm{d}_6$ (rf. Figure \ref{fig:tetrahedron}). Investigating Table \ref{tab:normal-vectors} we see both the first and the seventh row correspond the face $C_1$. Since the maximal value in the first row is larger than that in the seventh row, we choose the normal vector in the first row as the {\em effective normal vector}. In the computer algorithm we first identify the cluster of face normals that are close to each other (up to a threshold) and choose the one with the greatest maximal value as the effective normal vector. Table \ref{tab:effective} lists the effective face resulted from Table \ref{tab:normal-vectors}, and the face normals of the true polyheron as a comparsion.
\begin{table}[t]
  \centering
  \begin{tabular}{ccc}
    \toprule
    $n$ & $\bm{\nu}_j$ (recovered) & $\bm{\nu}_j$ (true) \tabularnewline
\midrule
                         $1$ & $(-0.85, -0.00, +0.53)$ & $(-0.82, 0.00, +0.58)$ \tabularnewline
                         $2$ & $(+0.85, 0.00, +0.53)$ & $(+0.82, 0.00, +0.58)$ \tabularnewline
                         $3$ & $(0.00, -0.85, -0.53)$ & $(0.00, -0.82, -0.58)$ \tabularnewline
                         $4$ & $(0.00, +0.85, -0.53)$ & $(0.00, +0.82, -0.58)$ \tabularnewline
\bottomrule
  \end{tabular}
  \caption{Recoverd effective face normals $\bm{\nu}_j$ and true face normals for the tetrahedron}
  \label{tab:effective}
\end{table}
\item Once the face normals are obtained, we compute the face areas $A_j$ using \eqref{eq:13}. The recovered and true face areas are listed in Table \ref{tab:area}.
\begin{table}[t]
  \centering
  \begin{tabular}{ccccc}
    \toprule
    \ & $A_1$ & $A_2$ & $A_3$ & $A_4$ \tabularnewline
\midrule
Recovered & 0.47 & 0.47 & 0.47 & 0.47 \tabularnewline
True & 0.43 & 0.43 & 0.43 & 0.43 \tabularnewline
\bottomrule
  \end{tabular}
  \caption{Recovered and true face areas for the tetrahedron.}
  \label{tab:area}
\end{table}
\end{enumerate}

Now the face normals and areas are recovered, we proceed to Step 2 to build the polyhedron.

\noindent {\bf Step 2: Build the polyhedron}

\begin{enumerate}
\item Recover the face offsets $\bm{\alpha}$ using \eqref{eq:18}. To this end, we need to construct the polyhedron for a given set of face normals $\bm{V}$ and face offsets $\bm{\alpha}$, and compute the face areas $a_j(\bm{V},\bm{\alpha})$. We implement this with the Qhull program \cite{BarberDobkinHuhdanpaa96}. Table \ref{tab:offsets} lists the resulted offsets after running the lsqnonlin program in MATLAB using default settings and initial guess $\bm{\alpha}=(1,1,1,1)$. 
  \begin{table}[t]
    \centering
    \begin{tabular}{ccccc}
\toprule
      \ & $\alpha_1$ & $\alpha_2$ & $\alpha_3$ & $\alpha_4$ \tabularnewline
\midrule
     Recovered & $0.21$ & $0.21$ & $0.21$ & $0.21$ \tabularnewline
     True & $0.20$ & $0.20$ & $0.20$ & $0.20$ \tabularnewline                                        
\bottomrule
    \end{tabular}
    \caption{Recovered and true face offsets for the tetrahedron}
    \label{tab:offsets}
  \end{table}
\item Once the face normals and offsets are recovered, we obtain the tetrahedron as the half space intersection of the face planes. We may also obtain the coordinates of the vertices using the Qhull program. Table \ref{tab:coordinates} lists the vertex coordinates for the recovered tetrahedron, as well as true tetrahedron for comparison. 
\begin{table}[t]
  \centering
  \begin{tabular}{ccc}
\toprule
    \ & Recovered & True \tabularnewline
\midrule
$P_1$ & $(+0.50, 0.00, -0.40)$ & $(+0.50, 0.00, -0.35)$ \tabularnewline
$P_2$ & $(-0.50, 0.00, -0.40)$ & $(-0.50, 0.00, -0.35)$ \tabularnewline
$P_3$ & $(0.00, +0.50, +0.40)$ & $(0.00, +0.50, +0.35)$ \tabularnewline
$P_4$ & $(0.00, -0.50, +0.40)$ & $(0.00, -0.50, +0.35)$ \tabularnewline
\bottomrule
  \end{tabular}
  \caption{Vertex coordinates for the recovered and true tetrahedron (centered at the origin).}
  \label{tab:coordinates}
\end{table}
\end{enumerate}

At this stage we have recovered the shape of the tetrahedron. Now we proceed to the recovery of the location of the tetrahedron. 

\noindent {\bf Step 3: Recover the location}

\begin{enumerate}
\item Send an incident wave of the form \eqref{eq:planewave} with a fixed polarization $\bm{p}=(0,0,1)$, propagation direction $\bm{d}=(1,0,0)$ and wavelength $\lambda=50$ (wavenumber $k=\pi/25$). Collect the far-field data $\bm{E}^\infty(\hat{\bm{x}}_j), \hat{\bm{x}}_j \in \mathcal{T}$, where $\mathcal{T}$ is a uniform grid of $1878$ points on $\Sphere^2$.
\item Determine the location of the polyhedron as a minimizer of the indicator function \eqref{eq:9} in the prescribed sample region $S=[0,100] \times [-100,100] \times [-100,100]$. The inner products in \eqref{eq:9} are computed in the same manner as in \eqref{eq:19}. Table \ref{tab:location} lists the result after using the fmincon program in MATLAB with the initial guess $(0,0,0)$. We observe the recovered location matches exactly with the true location up to the second digit.
\begin{table}[t]
  \centering
  \begin{tabular}{cc}
\toprule
   \ & Location \tabularnewline
\midrule
Recovered & $(50.00, 50.00, 50.00)$ \tabularnewline
True & $(50.00,50.00,50.00)$ \tabularnewline
\bottomrule
  \end{tabular}
  \caption{Vertex coordinates for the recovered and true tetrahedron (centered at the origin).}
  \label{tab:location}
\end{table}
\end{enumerate}

Once the location $\bm{z}$ is recovered, we simply translate the polyhedron recovered in step 1--2 so that its center of gravity moves to $\bm{x}=\bm{z}$. The final result is shown in Figure \ref{fig:tetrahedron_results} (B) (as a comparsion, the true polyhedron is shown in Figure \ref{fig:tetrahedron_results} (A)).
\begin{figure}[t]
  \centering
  \begin{subfigure}[b]{0.45\textwidth}
    \includegraphics[width=\textwidth]{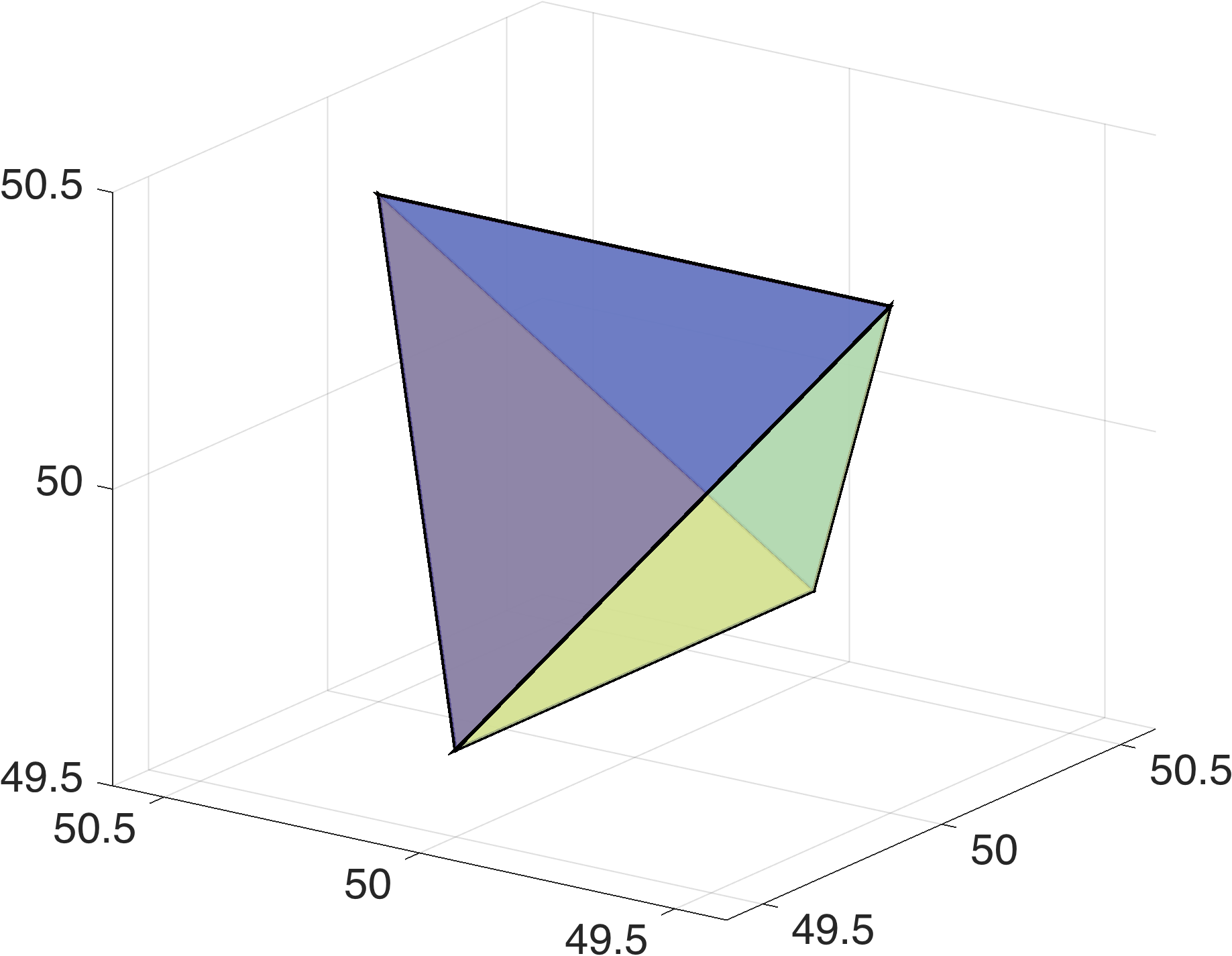}
    \caption{}
  \end{subfigure}
  \hfill
  \begin{subfigure}[b]{0.45\textwidth}
    \includegraphics[width=\textwidth]{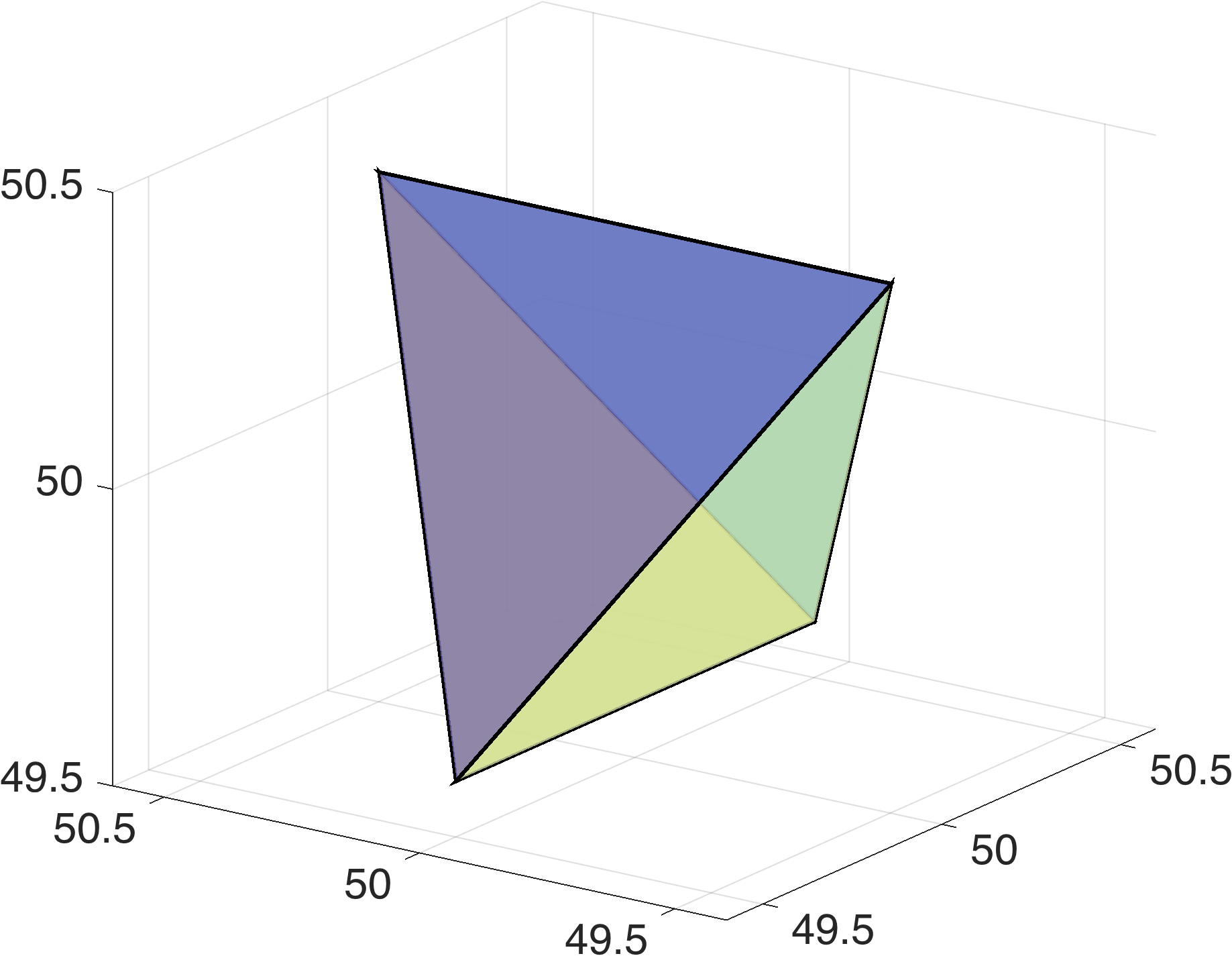}
    \caption{}
  \end{subfigure}
\begin{subfigure}[b]{0.45\textwidth}
    \includegraphics[width=\textwidth]{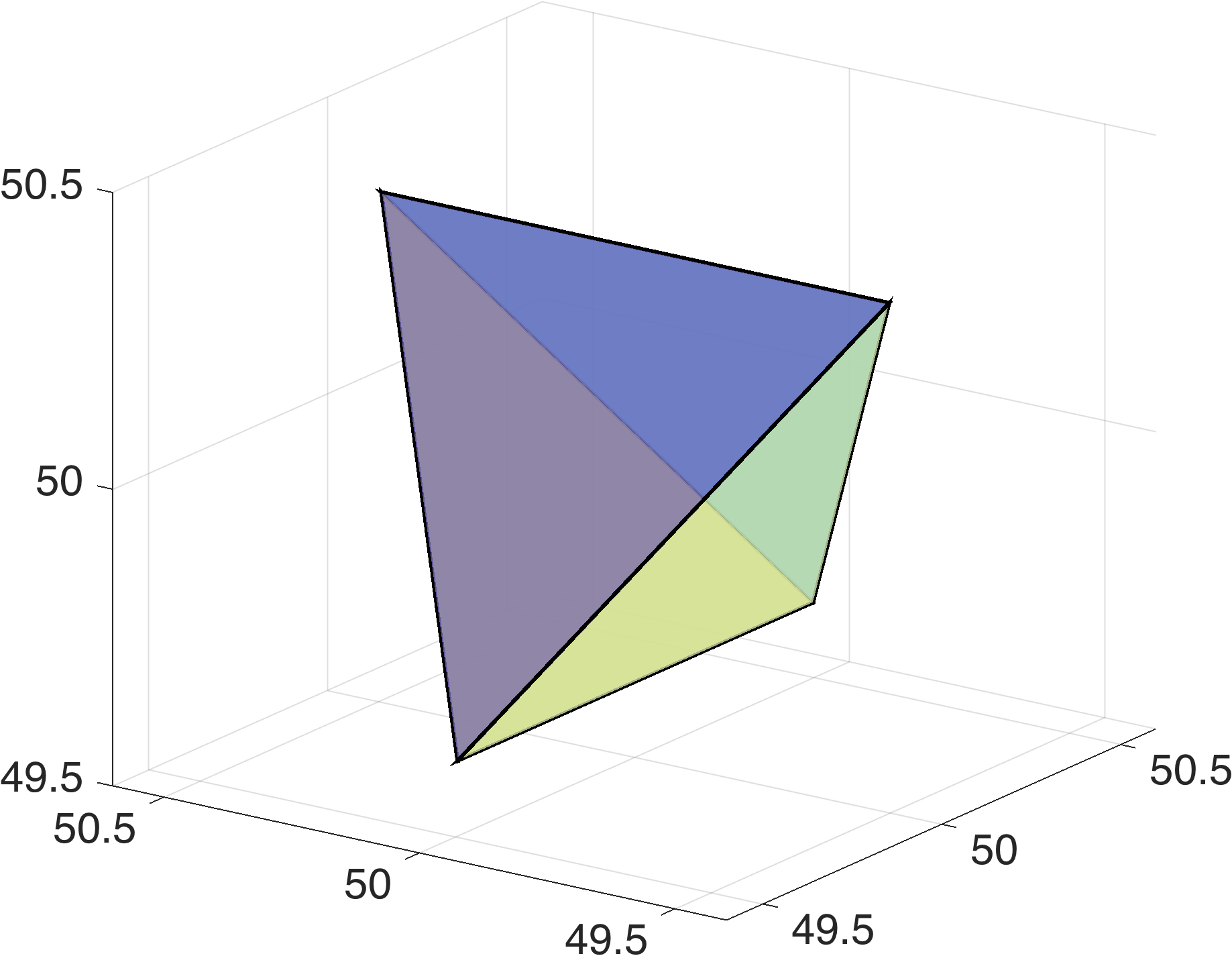}
    \caption{}
  \end{subfigure}
  \hfill
\begin{subfigure}[b]{0.45\textwidth}
    \includegraphics[width=\textwidth]{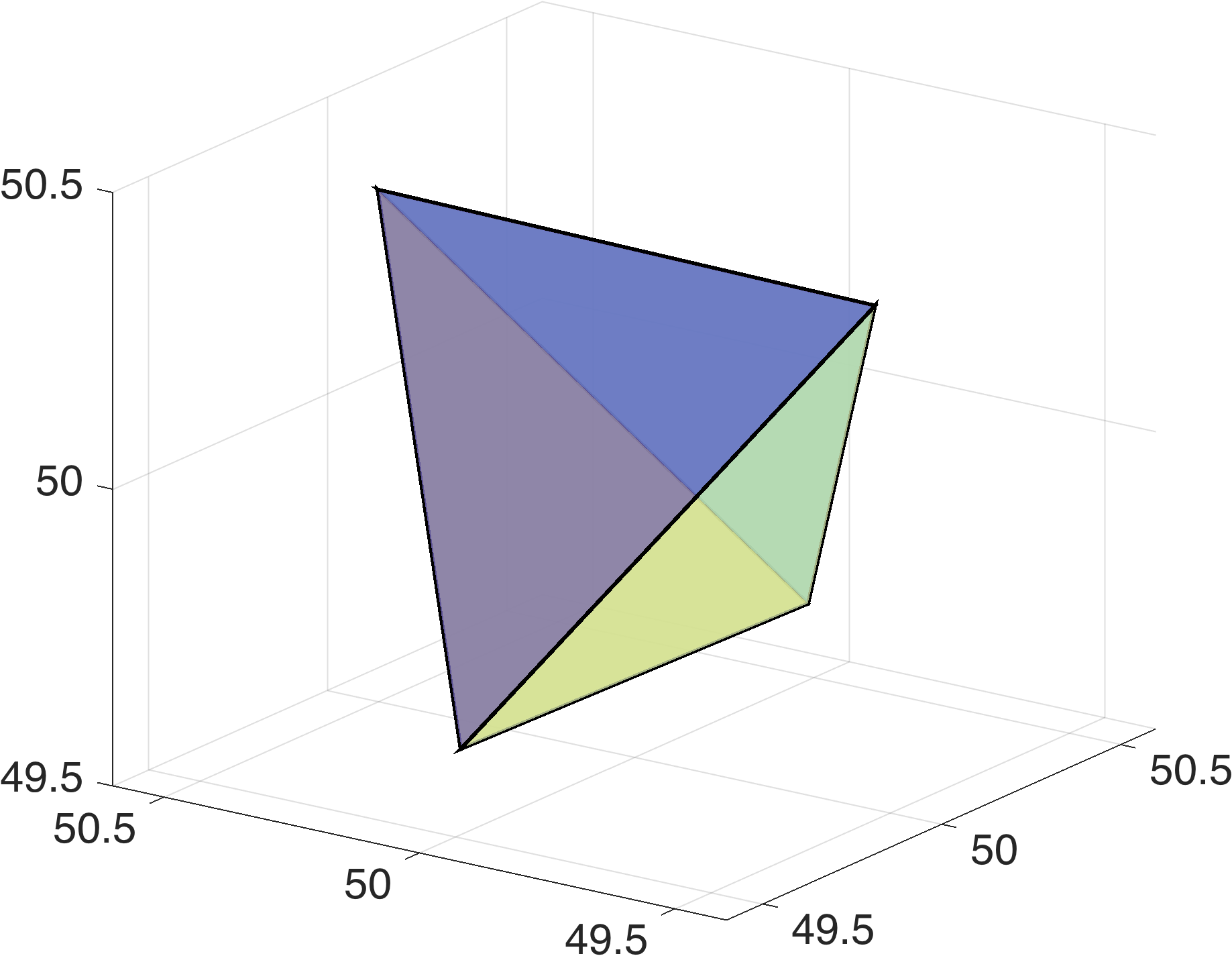}
    \caption{}
  \end{subfigure}
  \caption{Reconstruction results for the tetrahedron with noiseless measurement data. (A) true polyhedron; (B) recovered polyhedron with incident fields of wavelength $\lambda=0.5$ and noiseless measurements; (C) recovered polyhedron with incident fields of wavelength $\lambda=0.3$ and noiseless measurements; (D) recovered polyhedron with incident fields of wavelength $\lambda=0.3$ and noisy measurements (relative noise level $100\%$).}
  \label{fig:tetrahedron_results}
\end{figure}

Since the physical optics approximation and the approximations used in Theorem \ref{thm:main} become better as the wavenumber $k$ increases, we expect to obtain more accurate recovery in Steps 1--2 if we use incident fields of smaller wavelength. Figure \ref{fig:tetrahedron_results} (C) shows the recovered polyhedron when the incident fields used in step 1--2 have wavelength $\lambda=0.3$. Clearly the recovery is more accurate as compared with Figure \ref{fig:tetrahedron_results} (B), where the incident wavelength is $\lambda=0.5$. 

So far we assumed the measurement data is noiseless. In practice the measurement is never exact but contaminated with noise. To test the robustness of the scheme in terms of noise, we consider the following measurement data
\begin{align*}
  \left| \bm{E}^\infty_\delta(\hat{\bm{x}}_j) \right| = \left| \bm{E}^\infty(\hat{\bm{x}}_j) \right| (1 + \delta r_j)
\end{align*}
for Steps 1--2, where $\delta>0$ is the relative noise level and $r_j$ is a random number generated from the standardard normal distribution. The noisy measurement data for Step 3 is synthesized in a similar manner. Figure \ref{fig:tetrahedron_results} (D) shows the recovered polyhedron when the incident wavelength is $\lambda=0.3$ and the measurement data is contaminated with a noise of $100\%$ relative level. Comparing with Figure \ref{fig:tetrahedron_results} (C) we see the scheme is extremely robust in terms of measurement noise. This is due to the low-pass filter used in Step 1(3) and the inner product in the indicator function \eqref{eq:9} used in Step 3.

Next we test our scheme with more complicated polyhedra. We shall omit all the details and only show the final results. Figure \ref{fig:prism_results} shows the recovery results of a prism with three unit square faces and three equilateral triangular faces with $\lambda=0.5$ and noiseless measurements. The recovered prism is almost indistinguishable from the true prism.

Finally we consider the cuboctahedron shown in Figure \ref{fig:cuboctahedron_results} (A), which is obtained by truncating the eight corners of the unit cube at the middle points of the edges and rotating $45$ degress along the $z$ axis. The cuboctahedron consists of $12$ vertices and $14$ faces. The faces are smaller than the polyhedra considered previously, hence a smaller wavelength is needed to obtain satisfactory results. Figure \ref{fig:cuboctahedron_results} (B) shows the recovered polyhedron with $\lambda=0.3$ and noiseless measurements. In this example we observe a phenomenon which is inevitable in the reconstruction of polyhedra with errors, that is the change of adjacency relation. The recovered polyhedron consists of more vertices and edges than the true polyhedron. Some of the vertices in the polyhedron are splitted (marked in the red thick circle). Nevertheless the recovered polyhedron is still very close to the true one. We may also merge those splitted vertices (up to some threshold) if we wish.
\begin{figure}[t]
  \centering
  \begin{subfigure}[b]{0.45\textwidth}
    \includegraphics[width=\textwidth]{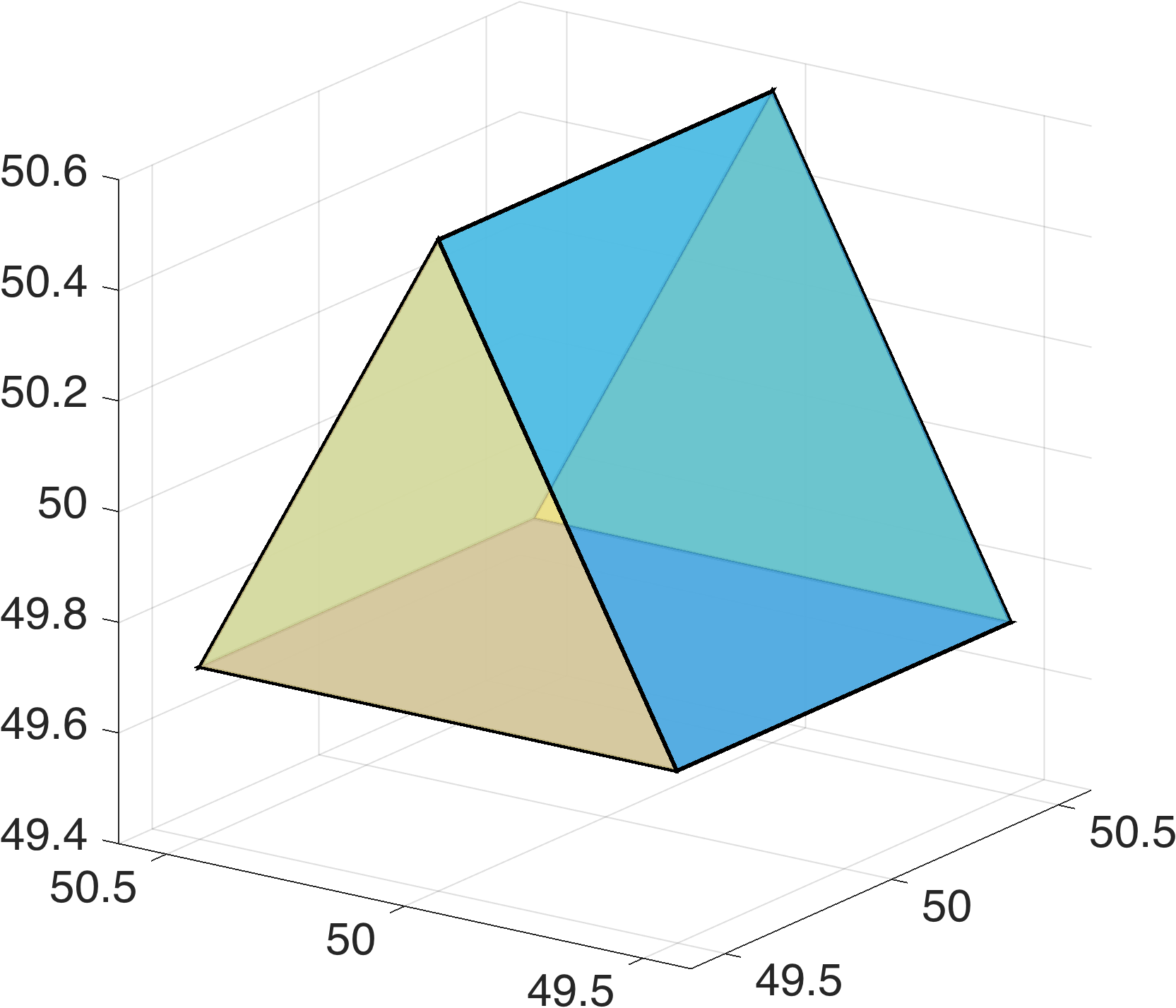}
  \end{subfigure}
  \hfill
  \begin{subfigure}[b]{0.45\textwidth}
    \includegraphics[width=\textwidth]{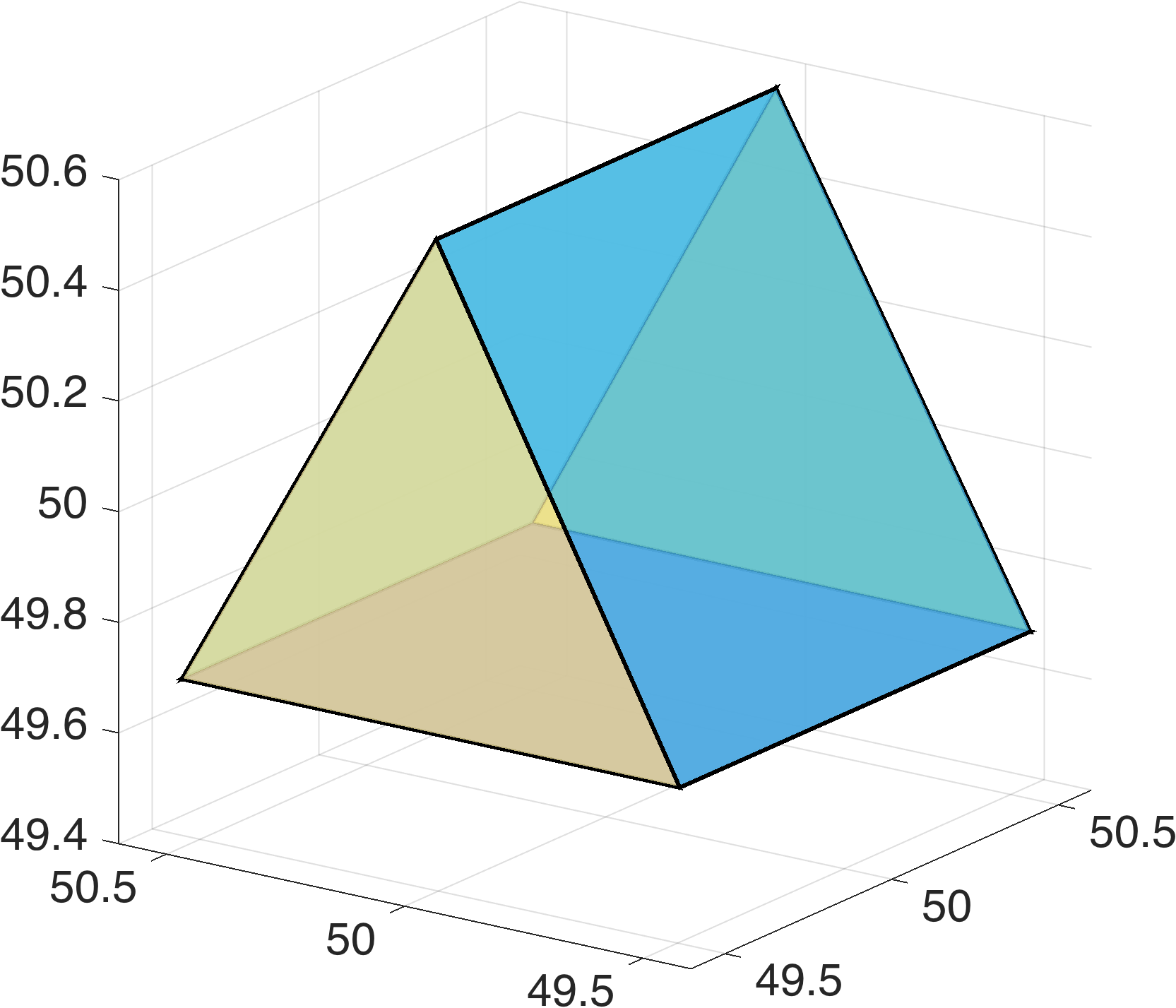}
  \end{subfigure}
  \caption{(A) The true prism; (B) the recovered prism with $\lambda=0.5$ and noiseless measurements.}
  \label{fig:prism_results}
\end{figure}
Figure \ref{fig:cube_results} shows the recovery results of a unit cube with $\lambda=0.5$ and noiseless measurements. The recoverd cube is slightly larger than the true cube.
\begin{figure}[t]
  \centering
  \begin{subfigure}[b]{0.45\textwidth}
    \includegraphics[width=\textwidth]{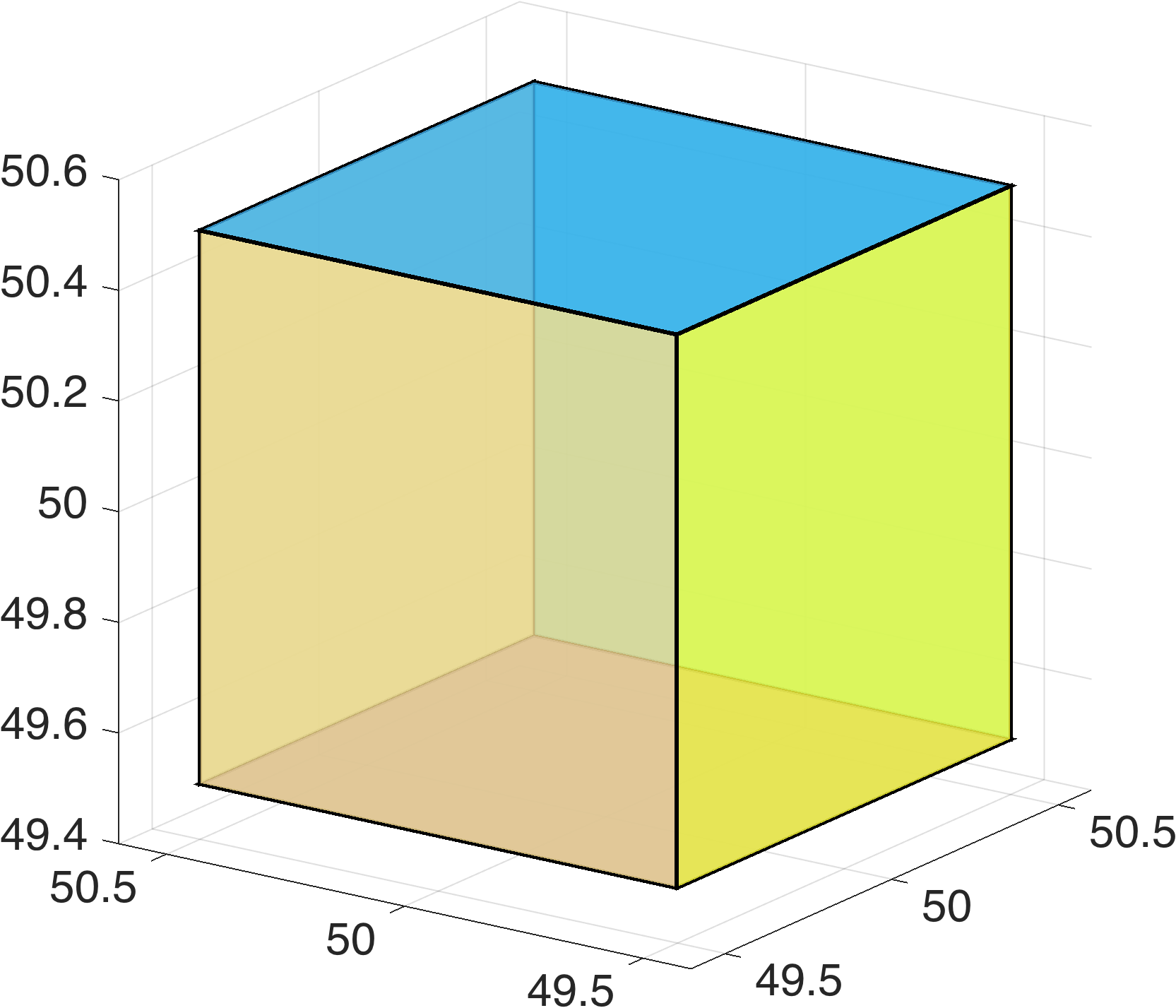}
  \end{subfigure}
  \hfill
  \begin{subfigure}[b]{0.45\textwidth}
    \includegraphics[width=\textwidth]{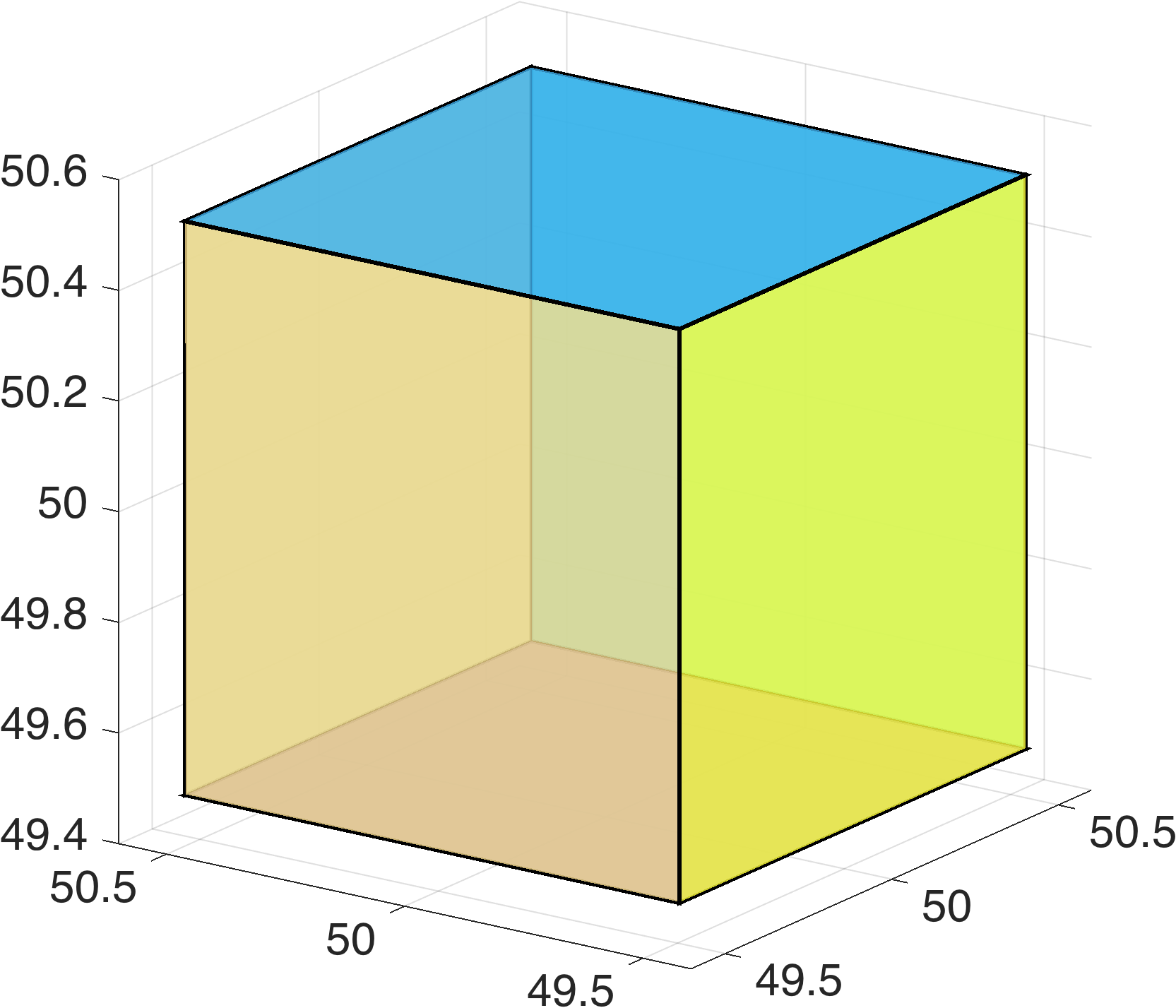}
  \end{subfigure}
  \caption{(A) The true cube; (B) the recovered cube with $\lambda=0.5$ and noiseless measurements.}
  \label{fig:cube_results}
\end{figure}
\begin{figure}[t]
  \centering
  \begin{subfigure}[b]{0.45\textwidth}
    \includegraphics[width=\textwidth]{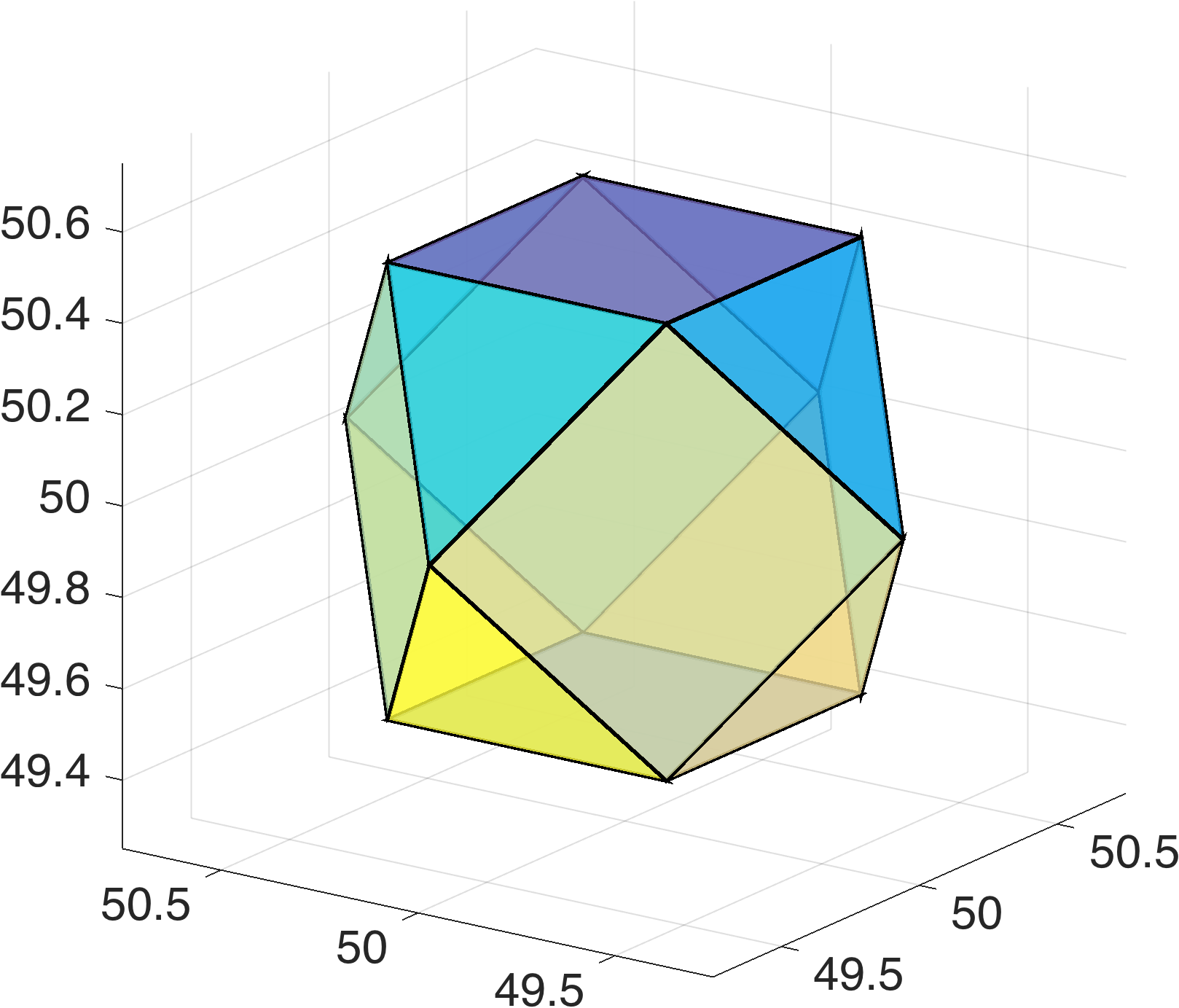}
  \end{subfigure}
  \hfill
  \begin{subfigure}[b]{0.45\textwidth}
    \includegraphics[width=\textwidth]{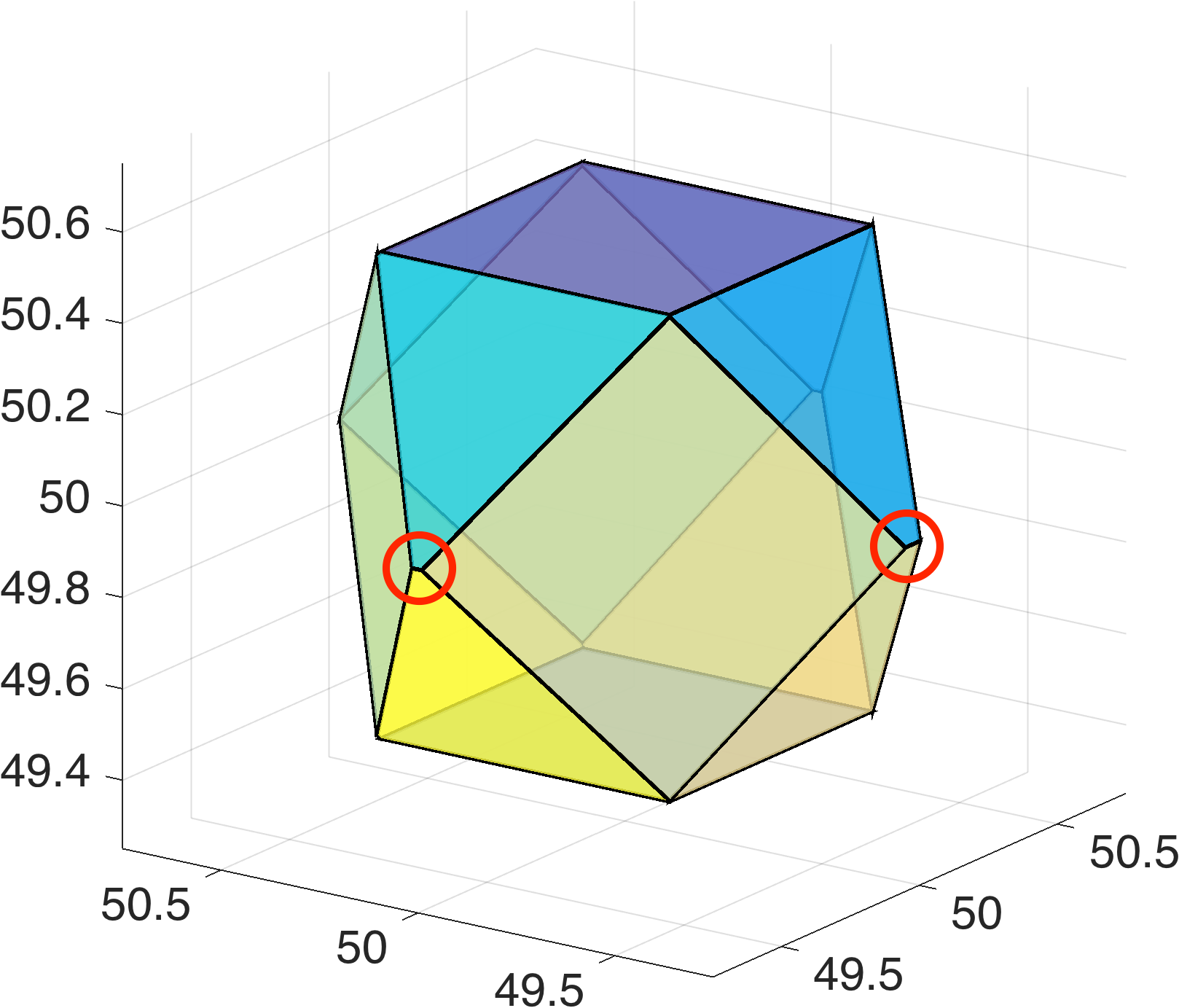}
  \end{subfigure}
  \caption{(A) the true polyhedron; (B) the recovered polyhedron with $\lambda=0.3$ and noiseless measurements.}
  \label{fig:cuboctahedron_results}
\end{figure}

\section{Conclusion}
We developed a novel scheme for solving an inverse electromagnetic scattering problem of recovering a convex polyhedron with a few phaseless and backscattering far-field measurements. The scheme consists of three major steps. The first step is to determine the face normals and face areas of the polyhedron. This is achieved by sending an incident field of a high frequency and collecting the phaseless far-field pattern, and theoretically supported by the local maxima behavior, which is proved based on the physical optics approximation. The second step is to reconstruct the polyhedron from the recovered face normals and areas. This is acccomplished by a simple least-square fitting method and an algorithm from computational geometry. The last step is to determine the location of the polyhedron by sending an incident field of a low frequency and collecting the far-field pattern. Numerical experiments show that the scheme is effective, fast and robust to measurement noise. This work is a significant extension of the recent work on 2D acoustic scattering to the much more challenging 3D electromagnetic scattering. For the future investigation, one may consider the recovery of non-convex polyhedra and inverse elastic scattering.

\end{document}